\documentclass[11pt, reqno]{article}

\usepackage{approximation}

\title{Ramified Approximation and Semistable Reduction}

\author{Xander Faber \\
IDA Center for Computing Sciences, Bowie, MD \\
xander@super.org}

\begin{document}
\maketitle

\begin{abstract}
  Let $K$ be a complete discretely valued field. An extension $L/K$ is
  \textbf{weakly totally ramified} if the residue extension is purely
  inseparable. We sharpen a result of Ax by showing that any Galois-invariant
  disk in the completion of an algebraic closure of $K$ contains an element
  that generates a separable weakly totally ramified extension. As an
  application, we prove that elliptic curves and dynamical systems on $\PP^1$
  achieve semistable reduction over a separable weakly totally ramified
  extension of the base field. We also obtain several arithmetic consequences
  for torsion points on elliptic curves and preperiodic points for dynamical
  systems.
\end{abstract}





\section{Introduction}

We set the following notation for this article:
\begin{center}
\begin{tabular}{lcl}
  $K$ & \hspace{1cm} & complete discretely valued field \\
  $v$ & & the discrete valuation on $K$, normalized so that $v(K^\times) = \ZZ$ \\
  $\tilde K$ & & residue field of $K$ \\
  $p$ & & characteristic exponent of $\tilde K$ (i.e., $p = \mathrm{char}(\tilde K) > 0$ or $p = 1$) \\
  $K_\alg$ & & fixed algebraic closure of $K$ \\
  $\CK$ & & completion of $K_\alg$ with respect to the unique extension of $v$ \\
  $G_K$ & & $\Aut_K(K_\alg) = \Aut_K(\CK)$, the group of field automorphisms fixing $K$ \\
\end{tabular}
\end{center}


For $a \in \CK$ and $r \in \Rinf$, define
\[
   D(a,r) = \{ b \in \CK \colon v(b-a) \ge r\}
\]
to be the disk with center $a$ and radius $r$. The group $G_K$ maps disks to
disks, and a disk $D$ is $G_K$-invariant if and only if there is an irreducible
polynomial $f \in K[z]$ whose roots belong to $D$. Given a $G_K$-invariant
disk, what can be said about the minimum-degree algebraic elements in it? A
first answer was given by Ax in 1970:

\begin{axlemma}[{\cite[Lemma~2]{Ax_zeros_of_polynomials}}]
  Let $f \in K[z]$ be a nonconstant polynomial of degree $d = qm$, where $q$
  and $m$ are coprime, and where $q$ is a power of $p$, the characteristic
  exponent of $\tilde K$. Assume that $D$ is a disk containing the roots of
  $f$. There exists $\alpha \in D$ such that $[K(\alpha) : K] \le q$.
\end{axlemma}

Following work of Tate \cite[\S3.3]{Tate_p_divisible_1966}, Ax was interested
in determining the $G_K$-invariant \textit{transcendental} elements of
$\CK$. Ax's Lemma implies that a minimum-degree element in a $G_K$-invariant
disk generates a $p$-power extension of $K$, but its proof tells us nothing about
the residue extension of the field generated by an element of minimum
degree. Our main result remedies this defect.

We say that a finite extension $L/K$ is \textbf{weakly totally ramified} if its
residue extension is purely inseparable. The properties ``totally ramified''
and ``weakly totally ramified'' agree when the residue field of $K$ is perfect,
so little is lost on a first reading by conflating these notions. We say that
$L/K$ is \textbf{purely wild} if it is weakly totally ramified and the
ramification index is a power of $p$; note that this includes the trivial case
$L=K$. (See Section~\ref{sec:definitions} for more on these definitions.)

\begin{approxthm}
  Let $\beta \in K_\alg$, and suppose that $D$ is a $G_K$-invariant disk
  containing $\beta$. There exists $\alpha \in D \cap K_\alg$ such that
  $K(\alpha)/K$ is separable and purely wild. Moreover, the
  following hold:
  \begin{itemize}
  \item(Inclusion of residue fields) We may take $\tilde K \subset \widetilde{K(\alpha)} \subset
    \widetilde{K(\beta)}$.
  \item(Inclusion of value groups) We may take
    $v(K^\times) \subset v\big(K(\alpha)^\times\big) \subset v\big(K(\beta)^\times\big)$.
  \end{itemize}
  In particular, if $K$ has residue characteristic zero, then $\alpha \in
  K$. If $K$ has residue characteristic $p > 0$, then $[K(\alpha) : K]$ is a
  power of $p$ that divides $[K(\beta) : K]$.
\end{approxthm}

The proof uses MacLane's method of approximants \cite{MacLane1,MacLane2} to
construct a sequence of semivaluations $V_1, \ldots, V_n$ on $K[z]$ satisfying
$V_1(f) < V_2(f) < \cdots < V_n(f) = \infty$, where $f$ is the minimal
polynomial of $\beta$. The first $V_j$ whose center lies in $D$ gives rise to
an element $\alpha$ with all of the desired properties. As a result, the
extension $K(\alpha)/K$ may be computed explicitly provided we have a model for
computation in the valued field $K$ and a means to factor polynomials over the
residue field. For example, this is feasible if $K$ is a finite extension of
the $p$-adic field $\QQ_p$ or of the function field $\FF_p(T)$. Other recent
applications of MacLane's method to geometric questions are given in
\cite{Obus_Srinivasan_MacLane,Obus_Wewers_MacLane}; both of these articles
address ramification questions, although of a different flavor than those in
the present study.

\begin{remark}
  Though the statement of Ramified Approximation does give inclusions of local
  invariants, it is typically not true that we have $K(\alpha) \subset
  K(\beta)$. For example, take $K = \QQ_2$ and $v = \ord_2$. If $\beta$ is
  given by
  \[
     \beta^{4} + 2 \beta^{3} + 4 \beta^{2} + 12 \beta + 12 = 0,
  \]
  then the proof produces an element $\alpha$ satisfying $\alpha^2 + 2 =
  0$. One verifies that $D = D(\beta,3/4)$ is the minimum disk containing all
  conjugates of $\beta$, that $\alpha \in D$, but that $K(\alpha) \not\subset K(\beta)$.
\end{remark}

Suppose the residue characteristic of $K$ is $p > 0$. The residue degree and
ramification index of a purely wild extension are both
powers of $p$. This means we can apply Ramified Approximation to the
irreducible factors of an arbitrary nonconstant polynomial in order to sharpen
Ax's Lemma, thus proving a Conjecture of Benedetto
\cite[Conj.~2]{Benedetto_potential_good_bounds}:

\begin{corollary}
  \label{cor:Ax_refinement}
  Let $f \in K[z]$ be a nonconstant polynomial of degree $d = qm$, where $q$
  and $m$ are coprime, and where $q$ is a power of $p$, the characteristic
  exponent of $\tilde K$. Assume that $D$ is a disk of finite radius containing
  the roots of $f$. There exists $\alpha \in D$ such that $[K(\alpha) : K] \le
  q$ and $K(\alpha)/K$ is separable and purely wild.
\end{corollary}


We now turn to a characterization of the minimum-degree elements in
$G_K$-invariant disks. The question is settled if $K$ has residue
characteristic zero: any such disk contains an element of $K$. The next result
is an immediate consequence of Ramified Approximation.

\begin{corollary}
  \label{cor:characterize}
  Let $K$ be a complete discretely valued field of residue characteristic $p >
  0$. Suppose that $D$ is a $G_K$-invariant disk.
  \begin{enumerate}
    \item[(a)] Let $L/K$ be a field extension of minimum degree such that $D
      \cap L \ne \varnothing$. Then $L/K$ is a purely wild
      extension.
    \item[(b)] Let $M/K$ be a field extension of minimum ramification index
      such that $D \cap M \ne \varnothing$. The ramification index of $M/ K$ is
      $p^m$ for some $m \ge 0$, and there exists a purely wild
      extension $M' / K$ with ramification index $p^m$ such that $D \cap M' \ne
      \varnothing$.
    \item[(c)] Let $N/K$ be a field extension of minimum residue degree such
      that $D \cap N \ne \varnothing$. The residue degree of $N/K$ is $p^n$ for
      some $n \ge 0$, and there exists a purely wild
      extension $N' / K$ with residue degree $p^n$ such that $D \cap N' \ne
      \varnothing$.
  \end{enumerate}
\end{corollary}




Corollary~\ref{cor:characterize}(b) sharpens a result of Rozensztajn when $p >
2$ \cite[Theorem~1.1.1]{Rozensztajn_2020}. Consequently, it has implications
for computing the complexity of standard subsets of $\PP^1(K_\alg)$ that arise
in the description of 2-dimensional crystalline representations of
$G_{\QQ_p}$. 

We now turn to an application of Ramified Approximation: weakly totally
ramified extensions of the base field are sufficient to achieve semistable
reduction for elliptic curves and dynamical systems on $\PP^1$.  The criterion
of N\'eron, Ogg, and Shafarevich suggests that ramification should be the only
obstruction to semistable reduction: if $K$ has perfect residue field, then an
elliptic curve $E_{/K}$ has good reduction if and only if its $\ell$-adic Tate
module is unramified for $\ell \ne \mathrm{char}(\tilde K)$
\cite[VII.7]{Silverman_AEC_2009}. Generalizations of this criterion for
semistable reduction of abelian varieties were proved by Serre--Tate and
Grothendieck \cite[\S7.4]{Bosch_et_al_Neron_Models_1990}. We are unaware of any
results along these lines in the setting of arboreal representations of
dynamical systems.

\begin{theoremA}
  Let $K$ be a complete discretely valued field, and let $E$ be an elliptic
  curve over $K$. There is a separable weakly totally ramified extension $L/K$ and a
  model for $E_L$ that admits semistable reduction. If $E$ has potential
  multiplicative reduction, then we may take $L/K$ to be at most quadratic. If
  $E$ has potential good reduction, then we may take $[L : K]$ to be
  \begin{itemize}
    \item a proper divisor of $12$ if $\mathrm{char}(\tilde K) \ne 2,3$;
    \item a divisor of $12$ if $\mathrm{char}(\tilde K) = 3$; or
    \item a divisor of $24$ if $\mathrm{char}(\tilde K) = 2$.
  \end{itemize}
\end{theoremA}

The novel feature of Theorem~A is that we may take $L/K$ to be weakly totally
ramified. This is quite easy to show when $\mathrm{char}(\tilde K) \ne 2$
or~$3$, a difficult exercise when $\mathrm{char}(\tilde K) = 3$, and requires
the power of Ramified Approximation when $\mathrm{char}(\tilde K) = 2$. The
proof gives an algorithm for finding the field extension $L/K$, and we suspect
that it produces the minimal such extension.

The word ``weakly'' is necessary in Theorem~A. For example, take $K =
\FF_3(T)\Ls{\pi}$ to be the field of formal Laurent series over $\FF_3(T)$ with
the valuation $\ord_\pi$, and define
\[
  E_{/K} \colon y^2 = x^3 + \pi x^2 - \pi x - T.
\]
Then $E$ has integral $j$-invariant, and an inseparable extension of the
residue field is required in order to $\pi$-adically separate the $2$-torsion
and achieve good reduction. For more examples, including some in residue
characteristic~2, see \cite{Szydlo}.


It is well known that if $m$ is coprime to the residue characteristic of $K$
and $E_{/K}$ is an elliptic curve with good reduction, then the reduction map
is injective on $m$-torsion: $E(K)[m] \hookrightarrow E(\tilde K)$. Since the
residue field does not grow when we pass to a totally ramified extension, this
observation allows us to make certain weak uniformity statements about
torsion. The following is an example of such a statement:

\begin{corollary}
  Let $F$ be a number field, and let $E_{/F}$ be an elliptic curve. Suppose that
  $E$ has potential good reduction at places $\pp$ and $\qq$ of $F$ such that
  $\pp$ lies above~2 and is totally ramified, and $\qq$ lies above~3 and is
  totally ramified.  Then
  \[
  \# E(F)_{\mathrm{tors}} \in \{1, 2, 3, 4, 5, 6,  12\}.
  \]
\end{corollary}

\begin{proof}
  By Theorem~A, there is a totally ramified extension of the completion $F_\pp$
  over which $E$ has good reduction. The residue field of such an extension is
  $\FF_2$.  An elliptic curve over $\FF_2$ has at most~5 rational points, so
  the prime-to-$2$ part of $E(F)_{\mathrm{tors}}$ is of size $1, 3$, or
  $5$. Similarly, there is a totally ramified extension of the completion
  $F_\qq$ over which $E$ has good reduction. An elliptic curve over $\FF_3$ has
  at most~7 rational points, so the prime-to-$3$ part is of size $1, 2, 4, 5$
  or $7$. The only possible torsion orders that fit these numerics are the ones
  given in the statement of the corollary.
\end{proof}

A dynamical system $f \colon \PP^1_K \to \PP^1_K$ --- i.e., a morphism of
degree at least~2 --- is said to have semistable reduction if its normalized
homogeneous resultant has minimal valuation among all possible coordinate
changes over $\CK$. If this minimal valuation is~0, then $f$ has good
reduction. There is an equivalent definition in terms of GIT
semistability. (See \S\ref{sec:dynamics} for references and more details.)
Rumely proved that the dynamical system $f$ admits semistable reduction after a
finite extension $L / K$ of degree at most $(\deg(f) + 1)^2$, and that $L/K$ is
either trivial or ramified \cite[\S3]{Rumely_Min_Res_Loc}. We strengthen both
of these statements.

Recall that $p$ is the characteristic exponent of the residue field $\tilde
K$. For an integer $n \ge 1$, define
\[
   q(n) = \begin{cases} p^{\ord_p(n)} & \text{ if $p > 1$} \\ 1 & \text{ if $p = 1$,} 
   \end{cases} \qquad
   q^*(n) = \max_{1 \le m \le n} q(m) =
   \begin{cases} p^{\lfloor \log_p(n)\rfloor} & \text{ if $p > 1$} \\ 1 & \text{ if $p = 1$.}
   \end{cases}
\]
Set $A(p,d) = (d+1) \ q^*(d+1)$, and set
\[
B(p,d) = \begin{cases}
  (d-1) \ q(d) & \text{ if $p > 1$, $d > 2$, and $d \equiv 0 \pmod p$} \\
  d \ q(d-1) & \text{ if $p > 1$, $d > 2$, and $d \equiv 1 \pmod p$} \\
  d + 1 & \text{ otherwise}.
\end{cases}
\]
Note that $A(p,d) \le (d+1)^2$, while $B(p,d) \le d(d-1)$ if $d > 2$. Both
quantities are at most $d+1$ when $K$ has residue characteristic zero or when
$p > d+1$.

\begin{theoremB}
  Let $K$ be a complete discretely valued field, and let $f \colon \PP^1_K \to
  \PP^1_K$ be a morphism of degree $d \ge 2$. There is a separable weakly
  totally ramified extension $L/K$ over which $f$ attains semistable
  reduction. Moreover, we may take $[L : K] \le A(p,d)$ in general, and we may
  take $[L : K] \le B(p,d)$ if $f$ has potential good reduction.
\end{theoremB}

Benedetto proved the numerical part of this result when $f$ has potential good
reduction \cite[Thm.~A]{Benedetto_potential_good_bounds}, and he conjectured
that one can choose the extension $L/K$ to be totally ramified when $K$ has
perfect residue field. Theorem~B extends and proves this conjecture. 


Similar to the case of elliptic curves, there are a number of weak uniformity
results that can be obtained for periodic points of dynamical systems over
number fields. \textit{Good reduction} versions of these were stated by Morton
and Silverman \cite{Morton_Silverman_1994}, and we can immediately extend all
of them to the case of \textit{potential} good reduction. For example:

\begin{corollary}
  Let $F$ be a number field and let $f \colon \PP^1_F \to \PP^1_F$ be a
  dynamical system of degree at least~2. Suppose that $P \in \PP^1(F)$ is
  periodic with minimal period $n$.
  \begin{itemize}
    \item[(a)] Let $\pp$ and $\qq$ be finite places of $F$ with distinct residue
      characteristics for which $f$ has potential good reduction. Then $n \le
      (N_{F/\QQ} \pp^2 - 1)(N_{F/\QQ} \qq^2 - 1)$.
    \item[(b)] If $F = \QQ$ and $f$ has potential good reduction at $2$ and $3$, or
      at $3$ and $5$, then $n \mid 24$.
  \end{itemize}
\end{corollary}

\begin{proof}
  For a given place of potential good reduction, we may apply Theorem~B to pass
  to a totally ramified extension at which $f$ attains good reduction. Since
  the residue field is unaffected by such an extension, the proofs of
  \cite[Cor.~B and Cor.~C]{Morton_Silverman_1994} apply.
\end{proof}

The typical approach to obtaining a semistable model for an elliptic curve is
to make a field extension that rationalizes certain torsion points
\cite[VII.5.4, A.1.4]{Silverman_AEC_2009}. The proof of Theorem~A uses a kind
of ``approximate rationalization'' technique: it suffices to nearly rationalize
torsion over a weakly totally ramified extension. The same idea works for
dynamical systems: one needs to approximately rationalize certain periodic
points. Ramified Approximation is the key ingredient behind this technique.
However, it is a fundamentally 1-dimensional tool, and so it is unclear if it
can be used to extend Theorems~A and~B to curves of genus at least~2, general
abelian varieties, or dynamical systems on $\PP^n_K$.

Some progress can be made if one places further restrictions on the field
$K$. Suppose the residue field of $K$ is quasi-finite --- i.e., $\tilde K$
admits a unique finite extension of every degree. Clark observed that an
abelian variety over $K$ attains semistable reduction over a totally ramified
extension of $K$ \cite{MO:17846}, though his argument does not control the
degree of the extension. The argument would apply to dynamical systems $f
\colon \PP^n_K \to \PP^n_K$ if we had an analogue of the N\'eron model that
behaved well under unramified extensions. 

\begin{conjecture}
  Let $K$ be a complete discretely valued field.
  \begin{itemize}
    \item Let $A_{/K}$ be an abelian variety. There is a separable weakly
      totally ramified extension $L/K$ and a model for $A_L$ that admits
      semistable reduction.
    \item Let $f \colon \PP^n_K \to \PP^n_K$ be a dynamical system of degree $d
      \ge 2$.  There is a separable weakly totally ramified extension $L/K$ and
      a model for $f_L$ that admits semistable reduction.
  \end{itemize}
\end{conjecture}

As a final remark, we note that for many applications, the field $K$ will not
be complete. Krasner's Lemma can be used to deduce algebraic analogues of
Ramified Approximation and Theorems~A and~B, with the same bounds for the
degrees of the necessary extensions.

We recall or prove various properties of discretely valued fields in
\S\ref{sec:definitions}; additional notational conventions will also be laid
out there.  We then discuss semivaluations on $K[z]$ from two different
vantages: as the points of the analytic affine line over $K$ in
\S\ref{sec:affine_line}, and as inductive valuations in MacLane's theory of
approximants in \S\ref{sec:maclane}. We turn to the proof of Ramified
Approximation in \S\ref{sec:approximation}. We prove Theorems~A and~B in
\S\ref{sec:elliptic} and \S\ref{sec:dynamics}, respectively.


\noindent \textbf{Acknowledgments.} While I was sure many years ago that
Theorems~A and~B are true, they would not have been proved without the wealth
of insight and encouragement I gained from talking with Andrew Obus, Rob
Benedetto, Bob Rumely, and David Zelinsky. My sincerest appreciation goes to
all of you.

\noindent \textbf{AI Declaration.} The author used Anthropic's Claude Opus 5 to
review a completed draft of this manuscript. The model identified typographical
errors and two substantive mathematical errors, for which it proposed
corrections. The author independently verified each proposed correction before
incorporating it into the manuscript. He takes full responsibility for the
final work.


\section{Preliminaries on complete discretely valued fields}
\label{sec:definitions}

We begin by defining several concepts related to complete discretely valued
fields and their extensions. These are well known when the residue field is
perfect, but there is some fussiness over inseparable residue extensions in the
imperfect case. Useful references include 
\cite[Ch.~II.6]{Neukirch} and \cite[Ch.~I.4]{Serre_Corps_Locaux}.

Let $K$ be a field that is complete with respect to a discrete
valuation~$v$. Without loss of generality, we may assume that $v(K^\times) =
\ZZ$. We write $\cO_K$ for the valuation ring, $\mm_K$ for the maximal ideal,
and $\tilde K = \cO_K / \mm_K$ for the residue field. If $a \in \cO_K$, we write
$\tilde a$ for the reduction of $a$ --- i.e., its image in the residue field.

Now let $L/K$ be a finite extension of $K$. Then $v$ extends {\em uniquely} to
a valuation on $L$, and $L$ is complete and discretely valued with respect to
the extended valuation. We will abuse notation by writing $v$ for this
extension. The degree of the residue field extension is denoted by $\ff(L/K)$,
and it satisfies $\ff(L/K) = \ff_s(L/K) \cdot \ff_i(L/K)$, where $\ff_s(L/K)$
and $\ff_i(L/K)$ are the separable and inseparable degrees of this extension,
respectively. Write $\ee(L/K)$ for the ramification index $[ v(L^\times) :
  v(K^\times)]$. As $K$ is complete, we have the following fundamental
equality
\[
\ee(L/K) \cdot \ff_i(L/K) \cdot \ff_s(L/K) = [L : K].
\]

\begin{definition}
  If $L/K$ is a finite extension of complete discretely valued fields, then $L/K$ is
  \begin{itemize}
  \item \textbf{weakly unramified} if $\ee(L/K) = 1$;
  \item \textbf{unramified} if $\ff_i(L/K) = \ee(L/K) = 1$;
  \item \textbf{weakly totally ramified} if $\ff_s(L/K) = 1$; and
  \item \textbf{totally ramified} if $\ff_s(L/K) = \ff_i(L/K) = 1$.
  \end{itemize}
  If $L/K$ is weakly totally ramified and $\ee(L/K)$ is a power of the
  characteristic exponent of $\tilde K$, then $L/K$ is \textbf{purely wild}.
\end{definition}

\begin{remark}
It is entirely standard that ``unramified'' means ``trivial ramification index
and separable residue extension''. Following
\cite[Tag~09E4]{stacks-project-8-22-2026}, a valued field extension is ``weakly
unramified'' if it is unramified when we ignore separability issues in the
residue field. We mimic this separation in the definitions of totally ramified
and weakly totally ramified. We caution the reader that not all authors use
``totally ramified'' in the strict sense that we do --- i.e., they allow purely
inseparable residue extensions in their definition.
\end{remark}

The following structural result will be used throughout the paper without
comment.

\begin{lemma}
  Let $M/L/K$ be finite extensions of complete discretely valued fields. Then
  $M/K$ is weakly totally ramified if and only if both $M/L$ and $L/K$ are, and
  similarly for purely wild extensions.
\end{lemma}

\begin{proof}
  The result follows immediately from multiplicativity of the quantities $\ff_s$
  and $\ee$ in towers.
\end{proof}

To address the discomfort that often accompanies working with inseparable
extensions, we show that they behave as one might hope with respect to residue
extensions and value groups, and that they may be approximated by separable
extensions. The latter is well known \cite{Ax_zeros_of_polynomials}, but we
give a simple constructive argument for those interested in explicit
calculation.

\begin{proposition}
  \label{prop:purely_insep}
  Let $K$ be a complete discretely valued field. Any finite purely inseparable
  extension of $K$ is purely wild.
\end{proposition}

\begin{proof}
  Let $L/K$ be a purely inseparable extension. Then $[L : K] = \ff(L/K)
  \ee(L/K)$, so that all of the ramification is wild. Let $\tilde L / \tilde K$ be the
  residue extension. Suppose that $\ff_s(L/K) > 1$. Then there is $\tilde
  \alpha \in \tilde L$ such that $\tilde K(\tilde \alpha)/\tilde K$ is nontrivial and
  separable. Write $\tilde g$ for the minimal polynomial of $\tilde \alpha$,
  and let $g \in K[z]$ be an irreducible polynomial of the same degree whose
  coefficients reduce to that of~$\tilde g$. By Hensel's lemma, $g$ has a root $\alpha
  \in L$. But $g$ has no multiple root, else $\tilde g$ does. It follows that
  $K(\alpha)$ is a nontrivial separable subextension of $L/K$, a contradiction.
\end{proof}


In order to choose the extension $L/K$ to be separable in the statement of
Ramified Approximation, we need to show that small perturbations of inseparable
polynomials are separable.

\begin{proposition}
  \label{prop:separable_dense}
  If $f \in K[z]$ is an irreducible polynomial with a multiple root, then the
  roots of the separable polynomial $f(z) + cz$ converge uniformly to those of
  $f(z)$ as $v(c) \to \infty$. In particular, the separable closure $K_\sep$ is
  dense in $K_\alg$. 
\end{proposition}

\begin{proof}
  We may assume $p > 0$ is the characteristic of $K$. Since $f$ has a multiple
  root, we may write $f(z) = g(z^q)$ with $q = p^r \ge p$ and some separable
  polynomial $g \in K[z]$.  For $c \in K^\times$, set
  \[
    h_c(z) = f(z) + c z.
  \]
  Let $\alpha \in K_\alg$ be a root of $f$. Write
  \[
     f(z) = a_q (z-\alpha)^q + a_{q+1}(z-\alpha)^{q+1} + \cdots,
   \]
  with $a_i \in K_\alg$. Then
  \[
    h_c(z + \alpha) = c \alpha + cz + a_q z^q + a_{q+1} z^{q+1} + \cdots.
  \]
  If $v(c)$ is sufficiently large, then the Newton polygon for $h_c(z +
  \alpha)$ shows that $h_c$ has $q$ roots of the form $\alpha + u$ with
  \[
  v(u) = \frac{1}{q} \left( v(c\alpha) - v(a_q) \right) \approx \frac{1}{q}v(c).
  \qedhere
  \]
\end{proof}

\begin{remark}
  Using MacLane's method of approximants as described in \S\ref{sec:maclane},
  one can show that the fields $K[z] / (f)$ and $K[z]/(f + cz)$ have isomorphic
  residue fields and identical value groups when $v(c)$ is sufficiently large.
\end{remark}

The following result was observed when $K$ is a finite extension of $\QQ_p$ in
\cite[Prop.~1.2.4]{Rozensztajn_2020}, though the proof is more general.

\begin{lemma}
  \label{lem:unramified}
  Let $L/K$ be an unramified extension of complete discretely valued
  fields. Suppose $D$ is a $G_K$-invariant disk containing an element of
  $L$. Then $D$ contains an element of $K$.
\end{lemma}

\begin{proof}
  Let $\pi$ be a uniformizer for $K$, and hence also for $L$. Let $S \subset
  \cO_L$ be a system of coset representatives for the residue field of $L$ such
  that $0 \in S$ and $s \in \cO_K$ if and only if $\tilde s \in \tilde K$, the
  residue field of $K$.  Write $D = D(a,r)$ for some $a \in L$, and write
  \[
  a = a_N \pi^N + a_{N+1} \pi^{N+1} + \cdots,
  \]
  where the $a_i \in S$ for all $i$, and $N \in \ZZ$. If all $a_i \in K$, then
  $a \in K$. Otherwise, let $n$ be the minimum index such that $a_i \in K$ for
  all $i < n$ and $a_n \not\in K$. Set $b = \sum_{i < n} a_i \pi^i \in K$. Then
  $v(a - \sigma(a)) \ge n$ for all $\sigma \in G_K$, with equality for any
  $\sigma$ that acts nontrivially on $a_n$. Now $n \ge r$ since $D$ is
  $G_K$-invariant, and it follows that 
  \[
    v(b - a) = v( a_n \pi^n + \cdots) = n \ge r.
   \]
  That is, $b \in D \cap K$. 
\end{proof}


\section{The analytic affine line over \texorpdfstring{$K$}{K} }
\label{sec:affine_line}

Berkovich defined the analytic affine line in terms of multiplicative seminorms
\cite[\S1.5]{Berkovich_Spectral_Theory_1990}. In order to connect this theory
with MacLane's work, we need to recast the analytic affine line in the language
of semivaluations. We also extend Berkovich's classification of points to the
non-algebraically closed setting.

For this paper, a \textbf{$K$-semivaluation} on $K[z]$ is a function $V \colon K[z]
\to \Rinf$ such that $V|_K = v$, and for all $f,g \in K[z]$, we
have
\[
V(fg) = V(f) + V(g) \qquad \text{and} \qquad V(f+g) \ge \min(V(f),V(g)).
\]
(We will use capital letters to denote semivaluations on a polynomial ring, to
distinguish from the valuation $v$ on the base field.) If, in addition, $V(f) =
\infty \Rightarrow f = 0$, then we will call $V$ a \textbf{$K$-valuation}.

For $f \in K[z]$ an irreducible polynomial, we define a $K$-semivaluation
$V_{f,\infty}$ by
\[
V_{f,\infty}(g) = v\left(g(\alpha)\right),
\]
where $\alpha \in K_\alg$ is a root of $f$. As $G_K$ acts on $K_\alg$ by
isometries, this definition is independent of the choice of $\alpha$. 

\begin{proposition}
  \label{prop:maximal}
  Let $V$ be a $K$-semivaluation on $K[z]$, and suppose there exists a nonzero
  polynomial $g$ such that $V(g) = \infty$. Then $V = V_{f,\infty}$, where $f$
  is a monic irreducible factor of $g$. 
\end{proposition}

\begin{proof}
  The set $\pp_V = \{g \in K[z] \colon V(g) = \infty\}$ is a nonzero prime ideal of
  $K[z]$. Let $f$ be the monic irreducible generator of $\pp_V$. Then $V$
  determines a valuation on the quotient field $L = K[z]/(f)$. As $V|_K = v$ and
  $K$ is complete, the valuation on $L$ is uniquely determined by $v$. But
  $V_{f,\infty}$ also determines a valuation on $L$, so by uniqueness, we have
  $V = V_{f,\infty}$.
\end{proof}

Consider $\Aff{K}$, the analytic affine line over $K$ in the sense of Berkovich
\cite[\S1.5]{Berkovich_Spectral_Theory_1990}. By definition, the points of
$\Aff{K}$ are the multiplicative seminorms on $K[z]$ that extend the norm $\|
\cdot \| := \exp(-v(\cdot))$ on $K$.  Taking $-\log(\cdot)$, we see that it is
equivalent to describe the points of $\Aff{K}$ as the $K$-semivaluations on the
polynomial ring $K[z]$. The topology is the weakest one such that every map $f\colon
\Aff{K} \to \RR \cup \{\infty\}$ given by $V \mapsto V(f)$ is continuous. (The
connected neighborhoods of $\infty$ are of the form $(t, \infty]$.)

Given a $K$-semivaluation $V$ on $K[z]$, the set $\pp_V = \{f \in K[z] \colon
V(f) = \infty\}$ is a prime ideal. The $K$-semivaluation $V$ defines a {\em
  valuation} on the domain $K[z] / \pp_V$, and hence also on its fraction
field. Define $\BLambda{V}$ to be the associated residue class field. (This is
Berkovich's notation; in MacLane's notation we have $\Lambda(V) = \BLambda{V}$.)

The space $\Aff{K}$ admits a partial order: we write $V \valle W$ if and only
if $V(f) \le W(f)$ for all $f \in K[z]$. We will write $V \vallt W$ to mean
that $V \valle W$ and $V \ne W$. Proposition~\ref{prop:maximal} shows that the
$K$-semivaluations $V_{f,\infty}$ are maximal for the partial order. 

\begin{remark}
  Note that $V \vallt W$ does not imply that $V(f) < W(f)$ for every polynomial
  $f$.
\end{remark}

Recall $\CK$ is the completion of an algebraic closure of $K$, and $G_K$ is the
group of automorphisms of $\CK$ that fix $K$. As $G_K$ acts by isometries for
the unique extension of the valuation to $K_\alg$, this allows us to extend $v$
to the completion $\CK$ of $K_\alg$. Since $v(K^\times) = \ZZ$, we have
$v(\CK^\times) = \QQ$.

Any $\CK$-semivaluation $\bar V$ on $\CK[z]$ can be restricted to a
$K$-semivaluation $K[z]$, from which we obtain a continuous map
\[
  \pr \colon \Aff{\CK} \to \Aff{K}.
\]
The map $\pr$ preserves the partial order. The $G_K$-action extends to the
affine line $\Aff{\CK}$, and we see that $\pr$ induces an isomorphism of
locally ringed spaces
\[
   \Aff{\CK} / G_K \simarrow \Aff{K}.
\]
(See \cite[Cor.~1.3.6]{Berkovich_Spectral_Theory_1990}.) In particular, since
$\Aff{\CK}$ is path-connected \cite[Cor.~1.14]{Baker-Rumely_BerkBook_2010}, so
is $\Aff{K}$. One can even show that $\Aff{K}$ is uniquely path-connected by
copying the proof for $\Aff{\CK}$.

Given a disk $D(a,r)$ inside $\CK$, the association
\[
   f \mapsto \inf_{b \in D(a,r)} v(f(b))
\]
defines a $\CK$-semivaluation. (This is a consequence of the maximum principle
in rigid geometry.) We write $\zeta_{a,r}$ for the point of $\Aff{\CK}$
corresponding to the infimum semivaluation on $D(a,r)$. The associated
semivaluation will be denoted $V_{a,r}$. (N.B. --- Our notation $\zeta_{a,r}$ is
nonstandard in that it is tailored to valuations.) Berkovich classified the
points of $\Aff{\CK}$ into four types \cite[\S2.1]{Baker-Rumely_BerkBook_2010}.

\newpage 

\begin{itemize}
\item \textbf{Type I}: $\zeta_{a,\infty}$ for some $a \in \CK$.
\item \textbf{Type II}: $\zeta_{a,r}$ for some $a \in \CK$ and $r \in \QQ$.
\item \textbf{Type III}: $\zeta_{a,r}$ for some $a \in \CK$ and $r \in \RR
  \smallsetminus \QQ$.
\item \textbf{Type IV}: A limit of points $(\zeta_{a_i,r_i})_{i \ge 0}$, where
  the associated sequence of disks $\left(D(a_i,r_i)\right)_{i \ge 0}$ is
  decreasing and has empty intersection.
\end{itemize}
A point $x$ is of type~I if and only if the prime ideal $\pp_x$ is nonzero. A
point $x$ is of type~II if and only if the associated residue field
$\BLambda{x}$ is of transcendence degree~1 over $\widetilde{\CK} = (\tilde K)_\alg$.

The type of a point is preserved under the action of the automorphism group
$G_K$, so we may extend Berkovich's classification to $\Aff{K}$.

\begin{definition}
  A point $x \in \Aff{K}$ is of \textbf{type~$j$} (with $j = $ I, II, III, or IV) if
  each point of $\pr^{-1}(x) \subset \Aff{\CK}$ is of type~$j$.
\end{definition}

We want to give a more intrinsic description of the points of
$\Aff{K}$ involving infimum semivaluations, as we have in the algebraically
closed setting. To that end, for each irreducible $\phi \in K[z]$
and each $s \in \Rinf$, define
\[
  D(\phi,s) = \{b \in \CK \colon v(\phi(b)) \ge s\}.
\]
Following R\"uth, we call these sets \textbf{diskoids}
\cite[\S4.4.1]{Ruth_Thesis_2014}. (If $s \in \QQ$ is finite, then $D(\phi,s)$
is a $K$-affinoid domain in the sense of rigid geometry.) If $s$ is
sufficiently small, then $D(\phi,s)$ is a disk containing the roots of
$\phi$. More generally, we now show that a diskoid is a $G_K$-orbit of disks.

\begin{lemma}
  Let $\phi \in K[z]$ be a nonconstant polynomial with roots $a_1, \ldots,
  a_n$, listed with multiplicity. Let $a$ be any root of $\phi$. The function
  $M_{\phi,a} \colon \Rinf \to \Rinf$ given by
  \[
     r \mapsto \sum_{i=1}^n \min\{r, v(a_i - a)\}
  \]
  is continuous, piecewise affine, strictly increasing, and bijective. 
\end{lemma}

\begin{proof}
  $M_{\phi,a}$ is a sum of functions that are continuous, piecewise affine, and
  non-decreasing, so it inherits all of those properties. Moreover, $a = a_j$
  for some $j$, so the function $\min(r, v(a_j -a)) = r$ is strictly
  increasing. It follows that $M$ is also strictly increasing. For $r$
  sufficiently small, we see $M_{\phi,a} = r \deg(\phi)$. Also, $M_{\phi,a}(r)
  \to \infty$ as $r \to \infty$. Hence,  $M_{\phi,a}$ is onto. 
\end{proof}

If $\phi \in K[z]$ is irreducible, $M_{\phi,a} = M_{\phi,b}$ for any distinct
roots $a,b$ of $\phi$. In that case, we will drop the root from the subscript
and simply write $M_\phi$.

\begin{proposition}
  \label{prop:disk_decomp}
  Let $\phi \in K[z]$ be irreducible and monic of degree $n$. Write $a_1,
  \ldots, a_n$ for the roots of $\phi$, counted with multiplicity. Let $r,s \in
  \Rinf$ be such that $M_{\phi}(r) = s$. Then
  \[
     D(\phi,s) = \bigcup_{i=1}^n D(a_i,r).
  \]
\end{proposition}

\begin{proof}
  If $r = s = \infty$, then the result is clear. In the remainder of the proof,
  we assume that $r,s \ne \infty$.

  Write $\phi(z) = \prod (z-a_i)$. Write $B = D(a_1,r)$. If $b \in B$, then we
  have
  \[
    v(\phi(b)) = \sum_{i=1}^n v(b - a_i) 
    \ge \sum_{\substack{1 \le i \le n\\ a_i \in B}} r + \sum_{\substack{1 \le i \le n\\ a_i \not\in B}} v(a_1 - a_i) 
    = \sum_{i=1}^n \min\{ r, v(a_1 - a_i)\} = s.
  \]
  That is, $b \in D(\phi,s)$. Since $D(\phi,s)$ is $G_K$-invariant, we find that $D(\phi,s) \supseteq \bigcup_i D(a_i,r)$.

  Now suppose that $b \not\in D(a_i,r)$ for any $i$. If necessary, re-index the
  roots of $\phi$ so that
  \[
    v(b-a_1) \ge v(b-a_2) \ge \cdots \ge v(b-a_n).
  \]
  As before, set $B = D(a_1,r)$. If $a_i \in B$, then $v(a_i - a_1) \ge r$, and
  we find
  \[
     v(b-a_i) < r = \min\{r, v(a_i-a_1)\}.
  \]
  If $a_i \not\in B$, then  $v(a_i - a_1) < r$, so that
  \[
     v(b-a_i) = \min\{v(b-a_i), v(b-a_1)\} \le v(a_i - a_1) = \min\{r, v(a_i - a_1)\}.
  \]
  Combining these observations yields
  \[
    v(\phi(b)) = \sum_{i=1}^n v(b - a_i) 
    < \sum_{i=1}^n \min\{r, v(a_i - a_1)\} = s.
  \]
  That is, $b \not\in D(\phi,s)$, and hence $D(\phi,s) \subseteq \bigcup_i D(a_i,r)$.
\end{proof}

\begin{proposition}
  \label{prop:preimage}
  Let $\phi \in K[z]$ be irreducible, and let $s \in \Rinf$. The map
  \[
     V_{\phi,s}(f) := \inf_{b \in D(\phi,s)} v(f(b))
     \]
is a $K$-semivaluation on $K[z]$. Write $\zeta_{\phi,s}$ for the corresponding
point of $\Aff{K}$. If
 \[
    D(\phi,s) = \bigcup_{i=1}^n D(a_i,r)
    \]
 as in Proposition~\ref{prop:disk_decomp}, then $\pr^{-1}(\zeta_{\phi,s}) =
 \{\zeta_{a_1,r}, \ldots, \zeta_{a_n,r}\}$.
\end{proposition}

\begin{proof}
  For $i=1, \ldots, n$, the $\CK$-semivaluations $\zeta_{a_i,r}$ agree when
  restricted to $K[z]$. Indeed, for $f \in K[z]$ and $\sigma \in G_K$, we have
  \begin{align*}
    V_{a_i,r}(f) &= \inf_{b \in D(a_i,r)} v(f(b)) \\
    &= \inf_{b \in D(a_i,r)} v(\sigma(f(b))) \\
    &= \inf_{b \in D(a_i,r)} v(f(\sigma(b))) \\
    &= \inf_{b \in \sigma(D(a_i,r))} v(f(b))  = V_{a_j,r}(f),
  \end{align*}
  where $\sigma(a_i) = a_j$. It follows that for $f \in K[z]$, we have
  \[
      \inf_{b \in D(\phi,s)} v(f(b)) = \inf_{\substack{1 \le i \le n \\ b \in D(a_i,r)}} v(f(b)) = V_{a_1,r}(f).
  \]
  In particular, $V_{\phi,s}$ is a $K$-semivaluation. We have also shown that
  $\{\zeta_{a_1,r}, \ldots, \zeta_{a_n,r}\} \subset
  \pr^{-1}(\zeta_{\phi,s})$.

  The opposite inclusion follows if we can show that $G_K$ acts transitively on
  the set of extensions of $V_{\phi,s}$ to $\CK[z]$. Apply
  \cite[Prop.~II.9.1]{Neukirch} to the infinite field extension $\CK(z) /
  K(z)$, where $V_{\phi,s}$ is the valuation on $K(z)$.
\end{proof}

Proposition~\ref{prop:preimage} allows us to give the desired classification of
points of $\Aff{K}$:
\begin{itemize}
\item \textbf{Type I}: $\zeta_{\phi,\infty}$ for some irreducible $\phi \in K[z]$.
\item \textbf{Type II}: $\zeta_{\phi,s}$ for some irreducible $\phi \in K[z]$
  and $s \in \QQ$.
\item \textbf{Type III}: $\zeta_{\phi,s}$ for some irreducible $\phi \in K[z]$
  and $s \in \RR \smallsetminus \QQ$.
\item \textbf{Type IV}: A limit of points $(\zeta_{\phi_i,s_i})_{i \ge 0}$,
  where the associated sequence of diskoids $\left(D(\phi_i,s_i)\right)_{i \ge
    0}$ is decreasing and has empty intersection.
\end{itemize}

\begin{remark}
  If $\phi$ is irreducible and $c \in K^\times$, then $\zeta_{c\phi,s} =
  \zeta_{\phi,s-v(c)}$. In particular, nothing is lost in the above
  classification of points by assuming that the polynomial $\phi$ is monic.
\end{remark}

We close this section with two facts that will be used repeatedly.

\begin{lemma}
  \label{lem:points_below}
  Let $\phi \in K[z]$ be monic and irreducible, and let $s \in \Rinf$.
  \begin{enumerate}
  \item \label{item:key_value} $V_{\phi,s}(\phi) = s$.
  \item \label{item:points_below} If $x \in \Aff{K}$ satisfies $x \valle
    \zeta_{\phi,s}$, then $x = \zeta_{\phi,s'}$ for some $s' \le s$.
  \end{enumerate}
\end{lemma}

\begin{proof}
  (\ref{item:key_value}) If $s = \infty$, then $D(\phi,\infty)$ is the set of
  roots of $\phi$, so that $V_{\phi,\infty}(\phi) = \infty$. Suppose instead
  that $s \in \RR$. The definition of the infimum semivaluation yields
  $V_{\phi,s}(\phi) \ge s$. Define $r = M_\phi^{-1}(s)$, and write $\phi(z) =
  \prod (z-a_i)$. Proposition~\ref{prop:disk_decomp} gives $D(\phi,s) =
  \bigcup_i D(a_i,r)$. If $r \in \QQ = v(K_\alg^\times)$, then there exists $b
  \in D(a_1,r)$ such that $v(b - a_j) = r$ for each $a_j \in D(a_1,r)$. This
  gives
  \[
     v(\phi(b)) = \sum_{i=1}^{\deg(\phi)} v(b - a_i) = \sum_{i=1}^{\deg(\phi)} \min \{r, v(a_1-a_i)\} = s.
  \]
  That is, $V_{\phi,s}(\phi) = s$. If instead $r \not\in \QQ$, then we may
  approximate $r$ by rational numbers and apply the preceding argument to
  produce $b \in D(a_1,r)$ such that $v(\phi(b)) \le s + \varepsilon$ for any
  $\varepsilon > 0$. So $V_{\phi,s}(\phi) = s$ in that case as well.

  (\ref{item:points_below}) Observe that if $\psi \in K[z]$ is monic and
  irreducible, if $u \in \Rinf$, and if $\zeta_{\psi,u} \valle \zeta_{\phi,s}$,
  then $D(\phi,s) \subseteq D(\psi,u)$. Indeed, for $b \in D(\phi,s)$, the 
  first part of the lemma gives 
  \[
  v(\psi(b)) \ge V_{\phi,s}(\psi) \ge V_{\psi,u}(\psi) = u.
  \]

  Now let $x \valle \zeta_{\phi,s}$. Suppose that $x$ is of type~IV, say $x$ is
  the limit of points $\zeta_{\phi_i,s_i}$ for which the diskoids
  $D(\phi_i,s_i)$ are decreasing with empty intersection. As the diskoids are
  decreasing, we have $\zeta_{\phi_i,s_i} \valle x \valle \zeta_{\phi,s}$ for
  every $i$, and the previous paragraph gives
  \[
     D(\phi,s) \subseteq \bigcap_i D(\phi_i,s_i) = \varnothing.
  \]
  This is absurd, since $D(\phi,s)$ contains the roots of $\phi$. So $x$ is of
  type~I, II, or~III, which means we may write $x = \zeta_{\psi,u}$ for some
  monic irreducible $\psi \in K[z]$ and some $u \in \Rinf$.

  By the observation above, $D(\psi,u)$ contains $D(\phi,s)$, and in particular
  it contains the roots of $\phi$. Write $D(\psi,u) = \bigcup_j D(b_j,t)$ as in
  Proposition~\ref{prop:disk_decomp}. Some $D(b_j,t)$ contains a root $a$ of
  $\phi$, and hence $D(b_j,t) = D(a,t)$. Proposition~\ref{prop:preimage} now
  shows that $\pr(\zeta_{a,t}) = x$. Propositions~\ref{prop:disk_decomp}
  and~\ref{prop:preimage} applied to $\phi$ show that $\pr(\zeta_{a,t}) =
  \zeta_{\phi,s'}$, where $s' = M_\phi(t)$. So $x = \zeta_{\phi,s'}$. In particular, 
  $\zeta_{\phi,s'} \valle \zeta_{\phi,s}$,  and now part~(\ref{item:key_value}) gives
  $s' = V_{\phi,s'}(\phi) \le V_{\phi,s}(\phi) = s$.
\end{proof}


\section{MacLane's method of approximants}
\label{sec:maclane}

Nearly a century ago, MacLane introduced an efficiently computable description
of an arbitrary $K$-semivaluation on the polynomial ring $K[z]$
\cite{MacLane1,MacLane2}. We give a very brief introduction to set
notation. MacLane does not require that $K$ be complete in his work, though our
discussion is greatly simplified by doing so. See \cite[\S2]{MacLane2} for full
definitions and a quick-start guide, or see \cite{MacLane1} for a more thorough
treatment.

The main object of study in MacLane's theory is the inductive
  semivaluation. A {\em first stage inductive semivaluation} $V_1$ arises from the
standard semivaluation on a Tate algebra: there is $a \in K$ and $\mu_1 \in \Rinf$
such that
\[
    V_1 \left(\sum c_m (z-a)^m \right) = \min_ m v(c_m) + m \cdot \mu_1
\]
for each polynomial $\sum c_m (z-a)^m \in K[z]$. We write $V_1 = [v,
  V_1(z-a) = \mu_1]$.

For $i > 1$, an {\em $i$th stage inductive semivaluation} $V_i$ is defined as
an {\em augmentation} of an $(i-1)$th stage valuation $V_{i-1}$: there is a
(monic) {\em key polynomial} $\phi_i$ and a {\em key value} $\mu_i >
V_{i-1}(\phi_i)$ such that
\[
   V_i \left(\sum c_m \phi_i^m \right) = \min_m V_{i-1}(c_m) + m \cdot \mu_i.
\]
Here $c_m \in K[z]$ is a polynomial of degree strictly less than $\phi_i$; any
polynomial $f \in K[z]$ admits a unique expansion of the form $\sum c_m
\phi_i^m$. We write $V_i = [V_{i-1}, V_i(\phi_i) = \mu_i]$.

\begin{remark}
  We have overloaded the bracket notation $W = [V, W(\phi) = \mu]$. If $V$ is a
  valuation on $K$, then $\phi$ must be linear and $W$ is a stage-1 valuation
  on $K[z]$. If $V$ is a valuation on $K[z]$, then $W$ is an augmentation of
  $V$ in MacLane's sense.
\end{remark}

Given a $K$-semivaluation $W$ on $K[z]$, there exists a sequence of inductive
valuations $V_1 \vallt V_2 \vallt V_3 \vallt \cdots$ such that $W = \lim V_i$
\cite[Thm.~8.1]{MacLane1}. In general, this sequence may be finite or infinite,
and the description is not entirely constructive. However, if $f$ is an
irreducible polynomial and $W = V_{f,\infty}$, then MacLane's {\em method of
  approximants} \cite{MacLane2} gives a construction. We briefly recall the main
steps:
\begin{enumerate}
\item Choose the first key polynomial to be $z - a$ for some $a \in K$. Write
  $f = \sum c_j (z-a)^j$ with $c_j \in K$. The lower convex hull of the points
  $(j, v(c_j))$ has a unique slope $m$; let the first key value be $\mu_1 =
  -m$. Set $V_1 = [v, V_1(z-a) = \mu_1]$.
\item Assume we have determined $V_i$, with key polynomial $\phi_i$ and key
  value $\mu_i$. If $\mu_i = \infty$, we are done. Otherwise, the residue field
  $\BLambda{V_i}$ is a rational function field $k_i(y)$, where $k_i$ is an extension
  of the residue field of $K$.
  \begin{itemize}
    \item Find a polynomial $g$ with $\deg(g) < \deg(\phi_i)$ such that
      $V_i(g) = -V_i(f)$. Let $\psi$ be an irreducible polynomial factor of
      the image of $gf$ in $\BLambda{V_i}$. 
    \item Let $\phi_{i+1}$ be a key polynomial for $V_i$ that is also a lift of
      $\psi$ in $K[z]$. (We may choose $\phi_{i+1}$ to be separable if we like
      since it is only well defined up to $V_i$-equivalence.)
    \item Write $f = \sum c_j \phi_{i+1}^j$, where $c_j \in K[z]$ and
      $\deg(c_j) < \deg(\phi_{i+1})$. The lower convex hull of the points $(j,
      V_i(c_j))$ has a unique slope $m$, and we take $\mu_{i+1} = -m$.
      Define $V_{i+1} = [V_i, V_{i+1}(\phi_{i+1}) = \mu_{i+1}]$. 
  \end{itemize}
\end{enumerate}

The method of approximants constructs a particular sequence of points $\zeta_1,
\zeta_2, \zeta_3, \ldots \in \Aff{K}$ that converges to
$\zeta_{f,\infty}$. Since $\Aff{K}$ is uniquely path-connected, each point
$\zeta_i$ can be thought of as a blaze on the trail leading to
$\zeta_{f,\infty}$. At each blaze, a new key polynomial $\phi$ is chosen such
that $\zeta_{\phi,\infty}$ lies in the same direction as
$\zeta_{f,\infty}$. One can show that the new key value $\mu$ is chosen so that
$\zeta_{f,\infty}$ and $\zeta_{\phi,\infty}$ do not lie in the same direction
at $\zeta_{\phi,\mu}$; that is, $\zeta_{\phi,\mu}$ is the maximal point $x \in
\Aff{K}$ such that $x \valle \zeta_{f,\infty}$ and $x \valle
\zeta_{\phi,\infty}$.





\subsection{Diskoids and inductive valuations}

Our primary goal in this section is to exhibit a correspondence between infimum
semivaluations on diskoids and inductive semivaluations.

\begin{proposition}
  \label{prop:general_ruth}
  Let $V$ be an inductive semivaluation on $K[z]$, with key
  polynomial $\phi$ and key value $\mu$. Then for each $f \in K[z]$,
  \[
     V(f) = \inf_{b \in D(\phi,\mu)} v(f(b)).
  \]
  That is, $V$ corresponds to the point $\zeta_{\phi, \mu} \in \Aff{K}$. 
\end{proposition}

R\"uth proved this result in the case where $\mu$ is rational --- i.e., $V$
corresponds to a type~II point of $\Aff{K}$
\cite[\S4.4.2]{Ruth_Thesis_2014}. Using the path connectedness of $\Aff{K}$, we
improve his argument and extend it to the case of irrational and infinite key
values. This will require some preparation.

We begin with an easier result showing that a point of $\Aff{K}$ of type~I, II,
or III can be represented by \textit{some} inductive semivaluation.

\begin{proposition}
  \label{prop:some_indval}
  Let $\zeta_{f,s}$ be a point of $\Aff{K}$ of type~I, II, or~III with $f$
  monic and irreducible. Then the corresponding semivaluation $V_{f,s}$ can be
  represented as an inductive semivaluation. If $s = \infty$, then we may take the
  final key polynomial to be $f$.
\end{proposition}

\begin{proof}
  We begin with the case of a type~I point $\zeta_{f,\infty}$. If $f$ is
  linear, say $f = z - a$ for $a \in K$, then $V_1 = [v, V(z-a) = \infty] =
  V_{f,\infty}$. If instead $\deg(f) > 1$, then the method of approximants
  gives a sequence of inductive semivaluations $V_1 \vallt V_2 \vallt V_3
  \vallt \cdots$ such that either $V_{f,\infty} = V_n$ for some $n \ge 2$, or
  else $V_{f,\infty} = \lim V_i$.  We treat these two cases separately.

  Suppose first that $V_{f,\infty} = V_n$ for some $n$.  Then $V_n(\phi_n) =
  \infty$, where $\phi_n$ is the $n$th key polynomial. This implies $f =
  \phi_n$ (Prop.~\ref{prop:maximal}).

  Next suppose that we have an infinite sequence of inductive valuations such
  that $V_{f,\infty} = \lim V_n$. For $n$ sufficiently large, this means $V_n$
  has the same value group as $K[z] / (f)$, and its residue field
  $\BLambda{V_n}$ is purely transcendental over the residue field of
  $K[z]/(f)$. It follows that $\deg(\phi_n) = \deg(f)$
  \cite[Thm.~9.3]{MacLane2}. (If $f$ is separable, an estimate can be given for
  the size of $n$ \cite[p.506]{MacLane2}.)  We claim that $f$ is a key
  polynomial for $V_n$. Indeed, $f$ is monic, so $f = \phi_n + c$ for some $c
  \in K[z]$ with degree strictly smaller than $\deg(\phi_n)$. By construction,
  $\mu_n = v(c)$, so $V_n(f) = V_n(c)$. And $f$ is clearly equivalence
  irreducible in the sense of \cite[p.494]{MacLane2} since the image of $f/c$
  in the residue ring is linear.  MacLane shows that $f$ is a key polynomial
  for $V_n$ \cite[Thm.~9.4]{MacLane1}. We conclude that $W = [V_n, W(f) =
    \infty]$ is a valid augmentation of $V_n$, and it equals $V_{f,\infty}$ by
  Proposition~\ref{prop:maximal}.

  Now assume that $\zeta_{f,s}$ is a type~II or~III point of $\Aff{K}$. Let $t
  \in \RR$ be the largest value such that the disk $D(z,t)$ contains
  $D(f,s)$. In particular, $\zeta_{z,t} \valle \zeta_{f,s} \vallt
  \zeta_{f,\infty}$. Take $V_1 = [v, V_1(z) = t]$, and apply what we have
  already shown about type~I points to represent $V_{f,\infty}$ as an inductive
  semivaluation $V_n$. This gives a sequence of inductive semivaluations $V_1
  \vallt V_2 \vallt \cdots \vallt V_n$ whose corresponding points $x_1 < x_2 <
  \cdots < x_n$ subdivide the segment $[\zeta_{z,t}, \zeta_{f,\infty}]$.

  If $V_j = V_{f,s}$ for some $j$, then we are finished. Suppose instead that
  $V_j \vallt V_{f,s} \vallt V_{j+1}$. Write $\phi = \phi_{j+1}$ for ease of
  notation. Since $\phi$ is a key polynomial for $V_j$, it follows that $W =
  [V_j, W(\phi) = u]$ is a valid augmentation of $V_j$ for all $u >
  V_j(\phi)$. Moreover, $W$ traverses the interval $(V_j, V_{j+1}]$ as $u$
  increases along the real interval $(V_j(\phi), \mu_{j+1}]$. It follows that
  there is $u$ in this interval such that $W = V_{f,s}$, as desired.
\end{proof}

\begin{lemma}
  \label{lem:degree_keyval_comp}
  Let $V_1 \vallt \cdots \vallt V_n$ be a sequence of augmented inductive
  semivaluations on $K[z]$. Write $\phi_i$ and $\mu_i$ for the key polynomial and
  key value of $V_i$, respectively. Then for each $i = 2, \ldots, n$, we have
  \[
      \mu_i \deg(\phi_{i-1}) > \mu_{i-1} \deg(\phi_i).
      \]
\end{lemma}

\begin{proof}
  Write
  \[
  \phi_i = c_m \phi_{i-1}^m + \cdots + c_1 \phi_{i-1} + c_0,
  \]
  where $\deg(c_j) < \deg(\phi_{i-1})$ for each $j$. Since $\phi_i$ is a key
  polynomial for $V_{i-1}$, we have $c_m = 1$ and $V_{i-1}(\phi_i) =
  V_{i-1}(\phi_{i-1}^m)$ \cite[Thm.~9.4]{MacLane1}. It follows that $m
  \deg(\phi_{i-1}) = \deg(\phi_i)$, and
  \[
     \mu_i = V_i(\phi_i) > V_{i-1}(\phi_i) = m \mu_{i-1}. \qedhere
  \]
\end{proof}

The next result is due to R\"uth \cite[Lem.~4.53]{Ruth_Thesis_2014}, though we
fill a small gap in the proof with Lemma~\ref{lem:degree_keyval_comp}.

\begin{lemma}
  \label{lem:descending_disks}
  Let $V_1, V_2, \ldots, V_n$ be a sequence of augmented inductive semivaluations
  on $K[z]$. Write $\phi_i$ and $\mu_i$ for the key polynomial and key value of
  $V_i$, respectively. Then
  \[
     D(\phi_n, \mu_n) \subsetneq \cdots \subsetneq D(\phi_2, \mu_2) \subsetneq
     D(\phi_1, \mu_1).
  \]
\end{lemma}

\begin{proof}
  Write $D_i = D(\phi_i, \mu_i)$ for each $i$.  Fix an index $i \ge 1$, and let
  us show that $D_{i+1} \subsetneq D_i$. For ease of notation, set $\phi =
  \phi_{i+1}$. We claim that the roots of $\phi$ lie inside $D_i$. Define $W =
      [V_i, W(\phi) = \infty]$. Then if $\alpha$ is a root of $\phi$, we have
  \[
   \mu_i = V_i(\phi_i) = W(\phi_i) = v(\phi_i(\alpha)),
  \]
  so that $\alpha \in D_i$. Since $D_i$ is $G_K$-invariant, we conclude that
  all roots of $\phi$ lie in $D_i$.

  Since $D(\phi, \infty) \subset D(\phi_i, \mu_i)$, we can write $D(\phi, s) =
  D(\phi_i, \mu_i)$ for some $s \in \RR$. Write
  \[
     D(\phi,s) = B_1 \cup \cdots \cup B_d,
  \]
   a union of $G_K$-translates of a disk of  radius $r =
  M_{\phi}^{-1}(s)$ (Prop.~\ref{prop:disk_decomp}). Set $r_1 = r$ and $r_j =
  v(B_1 - B_j)$ for $j = 2, \ldots, d$. (As $B_1$ and $B_j$ are disjoint, the
  difference between any pair of elements from $B_1$ and $B_j$ has the same
  valuation.) Each disk must contain $\deg(\phi)/d$ roots of $\phi$,
  counted with multiplicity. If we write $a_i$ for a root of $\phi$, then we
  have
  \begin{align*}
  s = M_\phi(r) &= \sum_{i=1}^{\deg(\phi)} \min \{r, v(a_1 - a_i)\} \\ 
  &= \sum_{j=1}^d \ \sum_{\substack{1 \le i \le \deg(\phi) \\ a_i \in B_j}} \min \{r, v(a_1 - a_i)\} \\
  &= \frac{\deg(\phi)}{d}(r + r_2 + \cdots + r_d).
  \end{align*}
  Applying this same argument to the decomposition $D(\phi_i,\mu_i) = B_1 \cup
  \cdots \cup B_d$ shows that
  \[
    \mu_i = \frac{\deg(\phi_i)}{d} (r + r_2 + \cdots + r_d).
  \]
  Combining the last two displayed equations, we find that
  \[
    s = \mu_i \frac{\deg(\phi)}{\deg(\phi_i)}.
    \]
  Then $s < \mu_{i+1}$ (Lem.~\ref{lem:degree_keyval_comp}), which means
  $D(\phi_i, \mu_i) = D(\phi,s) \supsetneq D(\phi,\mu_{i+1})$.
\end{proof}

\begin{proof}[Proof of Proposition~\ref{prop:general_ruth}]
  Let $V_1 \vallt \cdots \vallt V_n$ be a sequence of augmented inductive
  semivaluations on $K[z]$, where $V = V_n$. Write $\phi, \mu$ for the key
  polynomial and key value for $V_n$, respectively. We want to show that $V_n =
  V_{\phi,\mu}$.

  To begin, we claim that $V_{\phi,\mu} \valge V_n$. That is, if $b \in
  D(\phi,\mu)$, then $v(f(b)) \ge V_n(f)$ for all $f \in K[z]$.  Suppose first
  that $n = 1$. Write $V_1 = [v, V_1(z-a) = \mu]$ for some $a \in K$ and $\mu
  \in \Rinf$. For $f \in K[z]$, we write $f = \sum c_i (z-a)^i$ with $c_i \in
  K$. For $b \in D(z-a, \mu)$, we have
  \[
     v(f(b)) \ge \min_i \{ v(c_i) + i \cdot v(b-a)\} \ge \min_i \{v(c_i) + i \mu\} = V_1(f).
     \]

  Now suppose that $V_{\phi,\mu} \valge V_i$ for all indices $i < n$, and write
  $V_n = [V_{n-1}, V_n(\phi) = \mu]$. For $f \in K[z]$, we write $f = \sum c_j
  \phi^j$, where $c_j \in K[z]$ satisfies $\deg(c_j) < \deg(\phi)$. For any $b
  \in D(\phi,\mu)$, we find that
  \[
     v(c_j(b)) \ge V_{\phi_{n-1},\mu_{n-1}}(c_j) \ge V_{n-1}(c_j),
     \]
  where the first inequality is a consequence of $D(\phi_{n-1},\mu_{n-1})
  \supsetneq D(\phi,\mu)$ (Lem.~\ref{lem:descending_disks}), and the second is
  our inductive hypothesis. The ultrametric inequality gives
  \[
  v(f(b)) \ge \min_j \{ v(c_j(b)) + j \cdot v(\phi(b))\}
    \ge \min_j \{V_{n-1}(c_j) + j \mu \} = V_n(f).
  \]
  Since $b \in D(\phi,\mu)$ is arbitrary, we have $V_{\phi,\mu} \valge V_n$.

  Since $V_n \valle V_{\phi,\mu}$, Lemma~\ref{lem:points_below}(\ref{item:points_below})
  supplies some $s \le \mu$ with $V_n = V_{\phi,s}$. Part~(\ref{item:key_value})
  of the same lemma then gives
  \[
     s = V_{\phi,s}(\phi) = V_n(\phi) = \mu,
  \]
  since $\phi$ is the key polynomial of $V_n$ and $\mu$ is its key value. Hence
  $V_n = V_{\phi,\mu}$, and the proof is complete.
\end{proof}


\subsection{Local geometry at a type~II point}

\noindent \textbf{Convention.} In this section, we write $k$ for the residue
field of $K$ (instead of $\tilde K$). Then the residue field of $\CK$ is
$k_\alg$.

To motivate our need for local geometric considerations, we sketch the idea of
the proof of Ramified Approximation. Let $f \in K[z]$ be an irreducible
polynomial. Apply MacLane's method of approximants to obtain a sequence of
inductive semivaluations $V_1 \vallt V_2 \vallt \cdots \vallt V_n =
V_{f,\infty}$. This corresponds to a sequence of points $\zeta_1 \vallt \zeta_2
\vallt \cdots \vallt \zeta_n = \zeta_{f,\infty}$ of $\Aff{K}$, which in turn
corresponds to a descending sequence of diskoids $D_1 \supsetneq D_2 \supsetneq
\cdots \supsetneq D_n$, each of which contains the roots of $f$. The minimal
disk $D$ containing the roots of $f$ sits somewhere in this chain, say $D_j
\supseteq D \supsetneq D_{j+1}$. In particular, for $1 \le i \le j$, the
diskoid $D_i$ is actually a \textit{disk}, while for $i \ge j+1$, it is a
nontrivial $G_K$-orbit of disks. This bifurcation can be detected in the way
that ``tangent vectors'' collapse under the map $\pr \colon \Aff{\CK} \to
\Aff{K}$.

The notion of tangent space at a point of $\Aff{\CK}$ has been explored
thoroughly; for a quick summary, see \cite[\S2.2.5]{Faber_Berk_RamI_2013}. We
now extend these ideas to $\Aff{K}$. 

\begin{proposition}
  \label{prop:directions}
  Let $\zeta$ be a type~II point of $\Aff{K}$, represented by an inductive
  valuation $V$. Write ${\BLambda{V} \cong \ell(x)}$ for the residue field of
  $V$, where $\ell$ is a finite extension of $k$, the residue field of $K$.
  \begin{itemize}
    \item The places of $\ell(x)$ parameterize the connected components of
      $\Aff{K} \smallsetminus \{\zeta\}$.
    \item The finite places of $\ell(x)$ parameterize equivalence classes of
      key polynomials for $V$.
  \end{itemize}
\end{proposition}

\begin{proof}
 The infinite place of $\ell(x)$ corresponds to all valuations $W$ with $W \not
 \valge V$. Each other place of $\ell(x)$ is described by the order of
 vanishing at an irreducible monic polynomial $\psi \in \ell[x]$. Then $\psi$
 corresponds to a particular equivalence class of key polynomials $\phi$ for
 $V$ \cite[Thm.~13.1]{MacLane1}\footnote{The equivalence class corresponding to
 $\psi = x$ was excluded from \cite[Thm.~13.1]{MacLane1} in order to use
 \cite[Thm.~9.4]{MacLane1}, but \textit{loc.cit.} contains an error. It does
 not allow for $\phi_k$ to be a key polynomial for $V_k$. See
 \cite[Lem.~4.3]{MacLane2}.}, and equivalent key polynomials give the same germ
 of an augmentation. (Two augmentations $W,W'$ yield the same \textbf{germ} if
 there is an augmentation $W''$ of $V$ such that $W'' \valle W$ and $W'' \valle
 W'$.)  Finally, a germ at $V$ describes an equivalence class of augmentations
 $W$ with $V \vallt W$, and hence a connected component of $\Aff{K} \setminus
 \{\zeta\}$.
\end{proof}

Let $\zeta \in \Aff{K}$ be a point. Write $T_\zeta$ for the set of connected
components of $\Aff{K} \smallsetminus \{\zeta\}$; an element $\vec{w} \in T_\zeta$
will be called a \textbf{tangent vector} at $\zeta$. If we wish to emphasize that
it is a subset of $\Aff{K}$, we will write $U_{\vec{w}}$. 

Take $\alpha \in \Aff{\CK}$ to be a point, and let $\zeta = \pr(\alpha) \in
\Aff{K}$ be its image. A tangent vector $\vec{v} \in T_\alpha$ corresponds to a
connected component $U_{\vec{v}}$, and one sees that $\pr(U_{\vec{v}})$ is a
connected component of $\Aff{K} \smallsetminus \{\zeta\}$. It follows that
$\pr$ induces a well defined map $\pr_* \colon T_\alpha \to T_\zeta$. Given
$\vec{w} \in T_\zeta$, we define the \textbf{branch multiplicity} $m(\zeta,
\vec{w})$ to be the number of distinct vectors $\vec{v} \in T_\alpha$ such that
$\pr_*(\vec{v}) = \vec{w}$. As the notation suggests, this number is
independent of the choice of $\alpha \in \pr^{-1}(\zeta)$ because $G_K$ acts
transitively on this set and carries tangent vectors to tangent vectors. See
Figure~\ref{fig:branch_mult}.

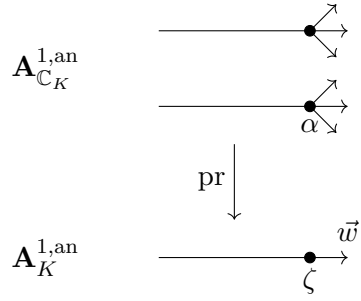
\begin{figure}[!htb]
  \begin{center}
\begin{tikzpicture}

  \node at (-1.5,0) {$\Aff{K}$};  
  \draw (0,0) -- (2,0);
  \filldraw (2,0) circle (0.07cm);
  \draw[->] (2,0) -- (2.5,0);
  \node at (2,-0.3) {$\zeta$};
  \node at (2.5, 0.35) {$\vec{w}$};

  \node at (-1.5,2.5) {$\Aff{\CK}$};    
  \draw (0,2) -- (2,2);
  \filldraw (2,2) circle (0.07cm);
  \node at (2,1.75) {$\alpha$};
  \draw[->] (2,2) -- (2.35,2.35);
  \draw[->] (2,2) -- (2.5,2);
  \draw[->] (2,2) -- (2.35,1.65);    
  
  \draw (0,3) -- (2,3);
  \filldraw (2,3) circle (0.07cm);
  \draw[->] (2,3) -- (2.35,3.35);
  \draw[->] (2,3) -- (2.5,3);
  \draw[->] (2,3) -- (2.35,2.65);

  \draw[->] (1,1.5) -- (1,0.5);
  \node at (.7,1) {$\pr$};
  
\end{tikzpicture}
  \end{center}
  \caption{An example showing the projection map locally collapsing
    $m(\zeta,\vec{w}) = 3$ branches at $\alpha \in \Aff{\CK}$ to a single
    branch at $\zeta \in \Aff{K}$. Here $\pr^{-1}(\zeta)$ consists of two points.}
  \label{fig:branch_mult}
\end{figure}

Keep the setup of the last paragraph. The map $\pr \colon \Aff{\CK} \to
\Aff{K}$ is a morphism of locally ringed spaces, so it induces a $k$-algebra
homomorphism of residue fields $\BLambda{\zeta} \to \BLambda{\alpha}$. If
$\alpha$ and $\zeta$ are type~II points, then in suitable coordinates this is a
homomorphism $\ell(x) \to k_{\alg}(y)$. To describe it precisely, we need some
additional terminology.

Write $W$ for the valuation corresponding to $\zeta$; we may assume that $W$ is
an inductive valuation (Prop.~\ref{prop:some_indval}). If $W$ is a first stage
inductive valuation, write $\Gamma = \ZZ$. Otherwise, $W$ is an augmentation of
another inductive valuation $V$, and we write $\Gamma$ for the value group of
$V$. If $f \in K[z]$ is a polynomial such that $W(f) \in \Gamma$, then there is
an equivalence unit $g \in K[z]$ such that $W(fg) = 0$ and $g g' \equiv 1 \pmod
V$ for some $g' \in K[z]$ \cite[Lem.~9.1/3]{MacLane1}. (Intuitively,
equivalence units are the polynomials that have no root in the direction at $V$
determined by $W$.) A \textbf{residual polynomial} for $f$ is the image of $fg$
in $\BLambda{W}$. It is well defined up to multiplication by an element of the
constant field of $\BLambda{W}$.

Let $\phi$ and $s$ be the key polynomial and key value for $W$,
respectively. Write $\tau \ge 1$ for the relative ramification index of $W/V$;
i.e., $\tau$ is the minimal positive integer such that $\tau s \in \Gamma$.
Then $W(\phi^\tau) \in \Gamma$, and we take an equivalence unit $g$ such that
$W(g \phi^\tau) = 0$. Define $x$ to be the image of $g \phi^\tau$ in
$\BLambda{W}$; it is our distinguished residual polynomial for
$\phi^\tau$. Then $\BLambda{W} \cong \ell(x)$, where $\ell = k_\alg \cap
\BLambda{W}$ \cite[Cor.~12.2]{MacLane1}.

Let $D(a,r)$ be a disk in the decomposition of $D(\phi,s)$
(Prop.~\ref{prop:disk_decomp}), so that $\pr(V_{a,r}) = V_{\phi,s} = W$. If $c
\in K_\alg$ is such that $v(c) = -r$, then $V_{a,r}(c(z-a)) = 0$, and we define
$y$ to be the image of $c(z-a)$ in $\BLambda{V_{a,r}}$. If $D(a,r)$ contains
exactly $d$ zeros of $\phi$, when counted with multiplicity, then the image of
$g \phi^\tau$ in $\BLambda{V_{a,r}}$ is $\eta^\tau$ for some polynomial $\eta$
of degree~$d$ in $y$. We have thus proved that there is a choice of coordinates such
that the composition
  \[
     \ell(x) \cong \BLambda{V_{\phi,s}} \longrightarrow \BLambda{V_{a,r}} \cong k_\alg(y)
  \]
is given by $x \mapsto \eta^\tau$, where $\eta \in k_\alg[y]$ is a particular
polynomial whose degree is the number of zeros of $\phi$ inside $D(a,r)$,
counted with multiplicity.

Now let $\psi \ne x$ be a monic irreducible polynomial, which corresponds to a
vector $\vec{w} \in T_\zeta$ (Prop.~\ref{prop:directions}). Since $\psi \ne x$
is irreducible, none of its roots is zero. Pushing $\psi$ through the above map
gives $\psi(\eta^\tau)$. The distinct roots of this polynomial correspond to
places of $k_\alg(y)$, which in turn correspond to tangent vectors at
$\zeta_{a,r}$ that map to $\vec{w}$. Hence $m(\zeta,\vec{w})$ is the number of
distinct roots of $\psi(\eta^\tau)$ in $k_\alg$.

Write $\tau^{(p)} = \tau / p^{\ord_p(\tau)}$ for the $p$-free part of
$\tau$. Define the \textbf{separable degree} of a nonconstant polynomial $h \in
k_\alg[y]$ as follows. If $k_\alg$ has characteristic zero, define
$\deg_\sep(h) = \deg(h)$. If instead the field has characteristic $p > 0$,
there is a unique way to write $h(y) = h_0(y^{p^m})$ with $m \ge 0$ maximal and
$h_0' \ne 0$; we set $\deg_\sep(h) = \deg(h_0)$. For any $c \in k_\alg$, the
equation $h(y) = c$ has at least one and at most $\deg_\sep(h)$ distinct
solutions.

Since $\psi$ is irreducible over $\ell$, it has precisely $\deg_\sep(\psi)$
distinct roots in $k_\alg$. For each of these roots $c$, the equation $y^\tau =
c$ has exactly $\tau^{(p)}$ distinct solutions $c'$, and for each such solution
$c'$, the equation $\eta(y) = c'$ has at least one and at most
$\deg_\sep(\eta)$ distinct solutions. We conclude that
\begin{equation}
  \label{eq:branch_multiplicity}
  \deg_\sep(\psi) \cdot \tau^{(p)} \le m(\zeta, \vec{w}) \le \deg_\sep(\psi)
  \cdot \deg_\sep(\eta) \cdot \tau^{(p)}.
\end{equation}
We record an important consequence of this formula.

\begin{proposition}
  \label{prop:disk_number_jump}
  Let $V_n$ be an inductive valuation with key polynomial $\phi_n$ and key
  value $\mu_n \in \QQ$. Let $\tau$ be the relative ramification index of $V_n$
  over $V_{n-1}$. Let $\phi \in K[z]$ be a key polynomial for $V_n$ such that
  $\phi \not\sim \phi_n \pmod{V_n}$.  For each $\mu > V_n(\phi)$, we have 
  $|\pr^{-1}(\zeta_{\phi,\mu}) | > |\pr^{-1}(\zeta_{\phi_n,\mu_n})|$ if one of the
  following holds:
  \begin{enumerate}
    \item\label{item:jump1}  A residual polynomial for $\phi \pmod {V_n}$ has at least two
      distinct roots in $k_\alg$, or
    \item\label{item:jump2} $\tau > 1$ and $\tau$ is not a power of the residue characteristic of $K$.
  \end{enumerate}
\end{proposition}

\begin{proof}
  Write $\zeta = \zeta_{\phi_n,\mu_n}$ and let $\alpha \in \pr^{-1}(\zeta)$.
  Let $\vec{w} \in T_\zeta$ be the tangent vector corresponding to the
  equivalence class of key polynomials containing $\phi$. We claim that the
  branch multiplicity $m(\zeta,\vec{w}) > 1$. Indeed, let $\psi$ be a residual
  polynomial for $\phi$. Then condition~\ref{item:jump1} implies
  $\deg_\sep(\psi) > 1$, while condition~\ref{item:jump2} implies $\tau^{(p)} >
  1$; the claim follows from the lower bound in \eqref{eq:branch_multiplicity}.

  Now let $\zeta' = \zeta_{\phi,\mu}$ for some $\mu > \mu_n$. Then $\zeta' \in
  U_{\vec{w}}$.  For each $\vec{v} \in T_\alpha$ such that $\pr_*(\vec{v}) =
  \vec{w}$, we have $\pr(U_{\vec{v}}) = U_{\vec{w}}$. It follows that
  \[
  \big|\pr^{-1}(\zeta')\big| \ge m(\zeta,\vec{w}) \ \big|\pr^{-1}(\zeta)\big|  > \big|\pr^{-1}(\zeta)\big|.
  \qedhere
  \]
\end{proof}


\section{Ramified Approximation}
\label{sec:approximation}

\noindent \textbf{Convention.} In this section, we write $k$ for the residue
field of $K$ (instead of $\tilde K$). Then the residue field of $\CK$ is
$k_\alg$.

We dispense with the (easier) case where $K$ has residue characteristic zero
first, and then we spend the remainder of the section on the case of positive
residue characteristic.

\begin{proposition}
  \label{prop:res_char_zero}
  Suppose the residue characteristic of $K$ is zero. Let
  \[
  f =c_n z^n + c_{n-1}z^{n-1} + \cdots + c_0
  \]
  be an irreducible polynomial in $K[z]$, and let $B \subset \CK$ be a disk
  containing the roots of $f$. Then $B$ contains the $K$-rational element
  $-c_{n-1} / (nc_n)$.
\end{proposition}

\begin{proof}
  If $a = a_1, \ldots, a_n$ are the distinct roots of $f$ in $K_\alg$, then
  \[
    f = c_n \prod (z - a_i) = c_n z^n - c_n (\sum a_i) z^{n-1} + \cdots,
  \]
  and $-c_{n-1} / (nc_n) = \frac{1}{n} \sum a_i \in K$. Write $b$ for this
  element.
  Let $r \in \RR$ be such that $B = D(a,r)$. Then
  \[
  v(b - a) = v \left( \frac{1}{n} \sum_{i=1}^n a_i - a\right)
  = v\left( \sum_{i=1}^n (a_i - a) \right) \ge \min_i v(a_i - a) \ge r.
  \]
  That is, $b \in B$. 
\end{proof}

Now we apply the method of approximants to locate purely wild elements inside
$G_K$-invariant disks.

\begin{theorem}
  \label{thm:WTR}
  Suppose the residue characteristic of $K$ is $p > 0$. Let $f \in K[z]$ be
  irreducible, and let $B \subset \CK$ be a disk containing the roots of $f$.
  Let $V_1 \vallt V_2 \vallt \cdots \vallt V_n = V_{f,\infty}$ be approximants
  to the $K$-semivaluation $V_{f,\infty}$, with key polynomials $\phi_1,
  \phi_2, \ldots, \phi_n = f$. Take $j \le n$ to be the maximum index such that
  \begin{enumerate}
  \item\label{item:WTR1} $\BLambda{V_j} \cap k_\alg$ is purely inseparable over $k$, and
  \item\label{item:WTR2} the ramification index of $V_i / v$ is a power of $p$ for all $i < j$.
  \end{enumerate}
  Then the roots of the key polynomial $\phi_j$ are contained in $B$, and $K[z]
  / (\phi_j)$ is purely wild.
\end{theorem}

\begin{remark}
  The same approach applies when the residue characteristic of $K$ is zero,
  though now we simply look for the largest index $j$ such that the key
  polynomial for $V_j$ is linear. If $\phi_j = z - b$, then $b \in B$ is the
  element we seek.
\end{remark}

\begin{proof}[Proof of Theorem~\ref{thm:WTR}]
  The result is evident if $f$ is purely inseparable
  (Prop.~\ref{prop:purely_insep}), so we may assume in what follows that $f$
  has at least two distinct roots. In particular, $B$ has finite radius.

  Suppose first that $j = n$ in the statement of the theorem. Since
  $\BLambda{V_n}$ is the residue field of the extension $K[z] / (f)$ and the
  ramification index is a power of $p$, the field $K[z]/(f)$ is purely wild,
  and the theorem holds. Therefore, we may assume for the remainder of the
  proof that $j < n$.

  For ease of notation, write $\phi = \phi_j$ and $\mu = \mu_j$, so that $V_j =
  [V_{j-1}, V_j(\phi) = \mu]$. Note that $\mu \in \QQ$ since $V_j$ is an
  approximant. To see that $K[z] / (\phi)$ is purely wild, consider the
  sequence of augmentations
  \[
     V_1 \vallt V_2 \vallt \cdots \vallt V_{j-1} \vallt W,
  \]
  where $W = [V_{j-1}, W(\phi) = \infty]$. Then $\BLambda{V_j} \cap k_\alg =
  \BLambda{W} \cap k_\alg$, and this extension depends only on the residue ring
  for $V_{j-1}$ and the key polynomial $\phi$ \cite[Thm.~12.1]{MacLane1}. In
  particular, since $\BLambda{W}$ is the residue field of $K[z] / (\phi)$, we
  see that it is purely inseparable over $k$. The ramification index of $V_{j-1} / v$ is
  a power of $p$, and it agrees with the ramification index of $K[z] / (\phi)$
  over $K$.

  It remains to show that the roots of $\phi$ lie inside the disk $B$. Define
  $D' = D(\phi,\mu)$.  If $D' \subset B$, then evidently $B$ contains the roots
  of $\phi$. So suppose $D' \not\subset B$. We see $D'$ and $D(\phi_1,\mu_1)$
  both contain the roots of $f$ because $V_1 \valle V_j \valle V_n =
  V_{f,\infty}$.  It follows that
  \[
     B \subsetneq D' \subseteq D(\phi_1,\mu_1),
  \]
  as we can write all three of these disks as $D(f,s)$ for various choices of
  $s$ (Lem.~\ref{lem:points_below}).
  Since $D'$ contains $B$, we see that $D(\phi_i,\mu_i)$ is a $G_K$-invariant
  disk for each $i \le j$. Now we use the maximality of the index $j$.

  Suppose first that condition~\ref{item:WTR1} fails for $j+1$. Then the field
  $\BLambda{V_{j+1}} \cap k_\alg$ has nontrivial separable degree over
  $\BLambda{V_{j}} \cap k_\alg$ by maximality of $j$, and
  $|\pr^{-1}(V_{\phi_{j+1},t})| > |\pr^{-1}(V_{\phi,\mu})|$ for all $t >
  V_j(\phi_{j+1})$ (Prop.~\ref{prop:disk_number_jump}). Note that $V_{j+1}$ and
  the infimum valuation on $B$ lie in the same direction from $V_j$ --- i.e.,
  the direction toward $V_{f,\infty}$. This would imply $B$ is a nontrivial
  Galois orbit of disks, which is a contradiction.

  Now suppose that condition~\ref{item:WTR2} fails for $j+1$, which means the relative
  ramification index of $V_j / V_{j-1}$ is not a power of $p$. We obtain a
  contradiction just as in the previous paragraph.

  Since either condition~\ref{item:WTR1} or~\ref{item:WTR2} must fail for the
  index $j+1$, we are forced to conclude that $D' \subseteq B$, and that the
  roots of $\phi$ are contained in $B$.
\end{proof}

\begin{example}
  \label{ex:not_approximant}
  We caution the reader that, with the setup of Theorem~\ref{thm:WTR}, the
  diskoid corresponding to the approximant $V_j$ may lie strictly inside the
  minimal disk about the roots of $f$. For example, take $K = \QQ_2$ and
  consider the polynomial
  \[
      f(z) = z^4 + 20z^2 + 292.
      \]
  The method of approximants gives $V_1 = [\ord_2, V_1(z) = 1/2]$, $V_2 = [V_1,
    V_2(z^2 + 2) = 4]$, and $V_3 = V_{f,\infty}$. Theorem~\ref{thm:WTR}
  tells us that $\sqrt{-2}$ lies inside $B$, the minimal disk about the roots
  of $f$. This minimal disk is $B = D(z^2 + 2, 3)$, while the diskoid for $V_2$
  is $D(z^2 + 2, 4)$. See Figure~\ref{fig:not_approximant}.
\end{example}

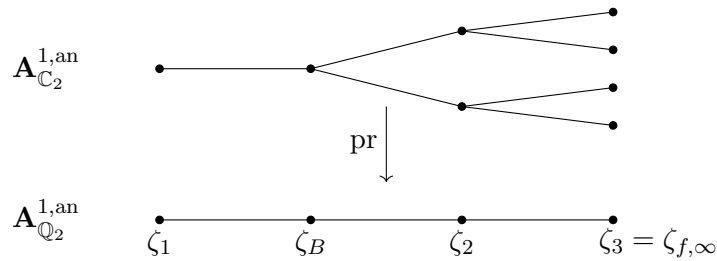
\begin{figure}[!htb]
  \begin{center}
\begin{tikzpicture}

  \draw (0,0) -- (6,0);
  \filldraw (0,0) circle (0.05cm);
  \filldraw (2,0) circle (0.05cm);
  \filldraw (4,0) circle (0.05cm);
  \filldraw (6,0) circle (0.05cm);
  \node at (0, -0.3) {$\zeta_1$};
  \node at (2, -0.3) {$\zeta_B$};  
  \node at (4, -0.3) {$\zeta_2$};
  \node at (6.6, -0.3) {$\zeta_3 = \zeta_{f,\infty}$};    
  \node at (-1.5,0) {$\Aff{\QQ_2}$};
  
  \draw (0,2) -- (2,2);
  \draw (2,2) -- (4,2.5);
  \draw (2,2) -- (4,1.5);
  \draw (4,2.5) -- (6,2.75);
  \draw (4,2.5) -- (6,2.25);
  \draw (4,1.5) -- (6,1.75);
  \draw (4,1.5) -- (6,1.25);  
  \filldraw (0,2) circle (0.05cm);
  \filldraw (2,2) circle (0.05cm);
  \filldraw (4,1.5) circle (0.05cm);  
  \filldraw (4,2.5) circle (0.05cm);
  \filldraw (6,2.75) circle (0.05cm);  
  \filldraw (6,2.25) circle (0.05cm);
  \filldraw (6,1.75) circle (0.05cm);  
  \filldraw (6,1.25) circle (0.05cm);
  \node at (-1.5,2) {$\Aff{\CC_2}$};  

  \draw[->] (3,1.5) -- (3,0.5);
  \node at (2.7,1) {$\pr$};

\end{tikzpicture}
  \end{center}
  \caption{A diagram of the positions of approximants for $f(z) = z^4 + 20z^2 +
    292$ in $\Aff{\QQ_2}$ and their pre-images in $\Aff{\CC_2}$. Here $\zeta_i$
    corresponds to the approximant $V_i$, while $\zeta_B$ corresponds to the
    infimum valuation on $B$. The minimum disk about the roots of $f$ is $D(z^2
    + 2,3)$; the corresponding point downstairs, $\zeta_B$ is not an
    approximant. See Example~\ref{ex:not_approximant}.}
\label{fig:not_approximant}
\end{figure}

\begin{proof}[Proof of Ramified Approximation]
  Suppose first that $K$ has residue characteristic zero. The result follows
  immediately from Proposition~\ref{prop:res_char_zero}.

  Now suppose that $K$ has residue characteristic $p > 0$. Taking $B = D$,
  Theorem~\ref{thm:WTR} gives a polynomial $\phi = \phi_j$ whose roots all lie
  in $D$. Let $\alpha \in K_\alg$ be a root of $\phi$. Then $K(\alpha) / K$ is
  purely wild. The construction of the approximants $V_1
  \vallt V_2 \vallt \cdots \vallt V_n = V_{f,\infty}$ gives a sequence of
  residue extensions $k \subset k_1 \subset k_2 \subset \cdots \subset
  \ell$, where the residue field of $K(\alpha)$ is $k_j$
  \cite[Thm.~12.1]{MacLane1}. This proves the claim about residue fields. A
  similar statement holds for value group extensions \cite[\S6]{MacLane2}.

  Finally, we note that in choosing key polynomials, we always have the
  flexibility to vary within an equivalence class. In particular, by adding an
  appropriate linear term $cz$ with $v(c)$ sufficiently large, we may assume
  that $\phi$ is separable (Prop.~\ref{prop:separable_dense}). That is, we may
  arrange for the extension $K(\alpha) / K$ to be separable.
\end{proof}

Our Ramified Approximation theorem, as stated, is sufficient to deduce
Theorem~B. In order to obtain some of the degree bounds in Theorem~A, we need
to look more closely at certain quartic polynomials. If $f \in K[z]$ is a
separable polynomial, we say that the roots of $f$ are \textbf{equispaced} if
$v(\alpha - \beta)$ is the same for any pair of distinct roots $\alpha,\beta$
of $f$.

\begin{lemma}
  \label{lem:equispaced}
  Suppose $K$ has residue characteristic~2, and let $f\in K[z]$ be a separable
  quartic polynomial with equispaced roots. Then either $f$ admits a weakly
  totally ramified root, or else there is $\alpha \in K_\alg$ such that
  \begin{itemize}
  \item $K(\alpha)/K$ is separable and weakly totally ramified of degree dividing~2;
  \item $\alpha$ lies in the minimal disk in $\CK$ containing the roots of $f$;
    and
  \item $v(\beta-\gamma) \in v(K(\alpha)^\times)$ for every distinct pair of
    roots $\beta, \gamma$ of $f$.
  \end{itemize}
\end{lemma}

\begin{proof}
  Assume that $f$ has no weakly totally ramified root. Without loss of
  generality, we may assume that $f$ is monic. Then one of the following must
  be true:
  \begin{enumerate}
    \item\label{item:equispaced1} $f = f_1 f_2$ for quadratic polynomials $f_1,
      f_2 \in K[z]$ and $K[z]/(f_i)$ is unramified for $i = 1, 2$;
    \item\label{item:equispaced2}  $f$ is irreducible and $K[z] / (f)$ is unramified;
    \item\label{item:equispaced3} $f$ is irreducible and $K[z] / (f)$ has
      ramification index~2 and separable residue degree $2$; or
    \item\label{item:equispaced4} $f$ is irreducible and $K[z] / (f)$ has
      separable residue degree~2 and inseparable residue degree~2.
  \end{enumerate}
  Let $D$ be the minimal disk about the roots of $f$. By Ramified
  Approximation, we may choose $\alpha \in D$ such that $K(\alpha) / K$ is
  separable and weakly totally ramified of degree dividing~2. All that remains
  to show is that the difference between any pair of roots has rational
  valuation over $K(\alpha)$.

  Suppose first that we are in case~\ref{item:equispaced1}
  or~\ref{item:equispaced2}. Then a root of $f$ generates an unramified
  extension, and so we may take $\alpha \in K$
  (Lem.~\ref{lem:unramified}). Define $g(z) = f(z+\alpha)$. Since the roots of
  $f$ are equispaced, the roots of $g$ must all have the same valuation $s$. If
  $s \not\in \ZZ$, then a root of $g$ generates a ramified extension of $K$, a
  contradiction.

  Now suppose that we are in case~\ref{item:equispaced3}
  or~\ref{item:equispaced4}. Let $V_1 \vallt V_2 \vallt \cdots \vallt V_n =
  V_{f,\infty}$ be a minimal representation of $V_{f,\infty}$ as an inductive
  semivaluation (Prop.~\ref{prop:some_indval}). We claim that $n = 3$, that
  $\phi_2$ is quadratic, and that the minimal disk $D$ containing the roots of
  $f$ satisfies $V_2(\cdot) = \inf_{x \in D}(\cdot)$.

  The key polynomials of a minimal representation have strictly increasing
  degrees \cite[Lem.~15.1]{MacLane1}, so $n = 2$ or $n = 3$. Since the
  separable degree of the residue extension of $K[z]/(f)$ is~2, there is an
  index $j$ such that a residual polynomial $\psi$ for the key polynomial
  $\phi_{j+1}$ has $\deg_\sep(\psi) = 2$.  Let $\tau$ be the relative
  ramification index for $V_j / V_{j-1}$, where $V_{j-1} := v$ if $j =
  1$. Since $p = 2$ and the ramification index of $K[z]/(f)$ is~$1$ or $2$, we
  find that $\tau^{(p)} = 1$. If $\zeta_j \in \Aff{K}$ corresponds to $V_j$ and
  $\vec{w} \in T_{\zeta_j}$ corresponds to (the equivalence class of)
  $\phi_{j+1}$, then the lower bound in \eqref{eq:branch_multiplicity} shows
  that $m(\zeta_j,\vec{w}) > 1$. Here $\eta$ is chosen as in the discussion
  before Proposition~\ref{prop:disk_number_jump} with $W = V_j$. Since the
  roots of $f$ are equispaced, this means $m(\zeta_j,\vec{w}) = 4$. In
  particular, the disk corresponding to $V_j$ is $D$.

  For the sake of a contradiction, suppose that $j = 1$. Then $\phi_1$ is a
  linear polynomial. The degree of $\eta$ equals the number of zeros of
  $\phi_1$ in $D$, which must be~1. Now \eqref{eq:branch_multiplicity} reads as
  $2 \le m(\zeta_1, \vec{w}) \le 2$. But we already showed above that
  $m(\zeta_1, \vec{w}) = 4$, so this is absurd. It follows that $j = 2$, which
  forces $n = 3$. We conclude further that the disk associated with $V_2$ is
  $D$.

  Now we know that $V_3 = V_{f,\infty}$ and that the key polynomial $\phi_2$
  for $V_2$ is quadratic. We take $\alpha$ to be a root of $\phi_2$. Let $r$ be
  the valuation of the difference between distinct roots of $f$. We apply the
  upper bound in \eqref{eq:branch_multiplicity} to obtain that $2
  \deg_\sep(\eta) \ge m(\zeta_2, \vec{w}) = 4$, which means $\eta$ has
  separable degree~2. Thus, $\phi_2$ has roots in exactly two distinct maximal
  open subdisks of $D = D(\alpha,r)$. That is, the valuation of the difference
  between the roots of $\phi_2$ is $r$. To conclude, write $\phi_2(z) = z^2 + b
  z + c$.  The two roots of $\phi_2$ are $\alpha$ and $-\alpha - b$, so the
  valuation of the difference is
  \[
     r = v(2\alpha+b) \in v(K(\alpha)^\times). \qedhere
  \]
\end{proof}
  

\section{Elliptic Curves}
\label{sec:elliptic}

We will give the proof of Theorem~A in three parts:
\begin{enumerate}
\item Residue characteristic different from~2 or~3 in \S\ref{sec:charnot23};
\item Residue characteristic different from~2 in \S\ref{sec:charnot2}; and
\item Residue characteristic~2 in \S\ref{sec:char2}.
\end{enumerate}
The first part is completely elementary.


\subsection{Residue characteristic different from~2 or~3}
\label{sec:charnot23}

\begin{theorem}
  Suppose that $K$ has residue characteristic not equal to~2 or~3. Let $E_{/K}$
  be an elliptic curve. Then $E$ attains semistable reduction over a separable
  totally ramified extension $L$ whose degree is a proper divisor of~12. If $E$
  has potential multiplicative reduction, we may take $[L : K] = 1$ or~$2$.
\end{theorem}

\begin{proof}
  Our hypothesis on the residue field ensures that $K$ itself has
  characteristic different from 2 or 3, and so we may assume that $E$ is given
  by a short Weierstrass equation of the form
  \[
    E \colon y^2 = x^3 + Ax + B,
    \]
  with $A,B \in K$. Define
  \[
     n = \min\left(3v(A),2v(B)\right).
  \]
  The discriminant $\Delta_E = -16(4A^3 + 27B^2)$ satisfies $v(\Delta_E) \geq
  n$.  Define $L = K(\pi^{1/g})$, where $g = 12 / \gcd(n,12)$ and $\pi$ is a
  uniformizer for $K$. Note that $L/K$ is separable and totally
  ramified. Moreover, since $n$ either vanishes, or else is divisible by at
  least one of~2 or~3, we find that the degree $[L : K]$ is in the set
  $\{1,2,3,4,6\}$.

  Set $u = \pi^{n/12} \in L$. Making the change of coordinates $x \mapsto u^2
  x$ and $y \mapsto u^3 y$, we obtain a new curve
  \[
    E'_{/L} \colon y^2 = x^3 + A'x + B',
  \]
  where
  \[
   v(A') = v(A) - 4v(u) = v(A) - \frac{n}{3} \geq 0 \quad \text{and} \quad
   v(B') = v(B) - 6v(u) = v(B) - \frac{n}{2} \geq 0.
   \]
   Hence $E'$ has coefficients in the valuation ring of $L$. The new
   discriminant satisfies
   \[
   v(\Delta_{E'}) = v(\Delta_E) - 12v(u) = v(\Delta_E) - n \geq 0.
   \]
  If $v(\Delta_E) = n$, then $E'$ has good reduction. Otherwise, we find that
  $v(\Delta_E) > n = 3v(A) = 2v(B)$. In particular, $v(A') = 0$, which implies
  that the reduction $\tilde E'$ has a nodal singularity
  \cite[III.1.4]{Silverman_AEC_2009}. Since $n = 3v(A) = 2v(B)$, we find that $6
  \mid n$, and hence $g = [L : K]$ divides~$2$.
\end{proof}


\subsection{Residue characteristic different from~2}
\label{sec:charnot2}

Throughout this entire section, we assume that $K$ has residue characteristic
different from~2. Our goal is to prove the following:

\begin{theorem}
\label{Thm:char-not-2}
Let $E_{/K}$ be an elliptic curve. There is a separable weakly totally ramified
extension $L/K$ of degree dividing~12 such that $E$ admits a semistable model
over $L$. If $E$ has potential multiplicative reduction, then we may take the
degree of $L/K$ to divide~2.
\end{theorem}

The next lemma gives a criterion for determining reduction type of an elliptic
curve in terms of the images of the 2-torsion in the residue field.

\begin{lemma}
  \label{lem:roots_reduction}
  Let $E_{/K}$ be an elliptic curve given by a Weierstrass equation of the form
  \[
  E \colon y^2 = (x-\alpha)(x-\beta)(x-\gamma),
  \]
  where $\alpha,\beta,\gamma \in K_\alg$. Suppose that $\min\left\{v(\alpha), v(\beta), v(\gamma)\right\} \ge 0$. 
  \begin{enumerate}
  \item If the reductions $\tilde \alpha$, $\tilde \beta$, and $\tilde \gamma$
    are pairwise distinct, then $E$ has good reduction.
  \item If exactly two of the reductions $\tilde \alpha$, $\tilde \beta$, and
    $\tilde \gamma$ coincide, then $E$ has multiplicative reduction.
  \item If all three of the reductions $\tilde \alpha$, $\tilde \beta$, and
    $\tilde \gamma$ coincide, then $E$ has additive reduction.
  \end{enumerate}
  In particular, this model of $E$ has semistable reduction if and only if $D(0,0)$ is the
  minimal disk containing $\alpha, \beta, \gamma$.
\end{lemma}

\begin{remark}
The different reduction types correspond to different group structures on the
special fiber of the N\'eron model $\cE$. (In the setting of the lemma, it
coincides with the locus of smooth points of $E$, viewed as a scheme over
$\cO_K$.) Assume the residue field of $K$ is algebraically closed for
simplicity. If $E$ has good reduction, then the 2-torsion remains separate
under reduction to yield the group
\[
\cE_s[2] \cong \ZZ/2\ZZ \times \ZZ/2\ZZ.
\]
If $E$ has multiplicative reduction, then two of the 2-torsion points collapse
to a singular point, and we have
\[
\cE[2] \cong \GG_m[2] \cong \ZZ/2\ZZ.
\]
If $E$ has additive reduction, then all of the nontrivial 2-torsion collapses to
a singular point, and we have
\[
\cE_s[2] \cong \GG_a[2] = 0.
\]
\end{remark}

\begin{proof}[Proof of Lemma~\ref{lem:roots_reduction}]
  The discriminant of $E$ is given by
  \[
  \Delta_E = 2^4(\alpha - \beta)^2(\alpha - \gamma)^2(\beta - \gamma)^2.
  \]
  If $\alpha$, $\beta$, $\gamma$ have distinct images in the residue field, then
  $\tilde \Delta_E = \Delta_{\tilde E} \neq 0$. Hence $E$ has good reduction.

  If any of $\tilde \alpha$, $\tilde \beta$, and $\tilde \gamma$ coincide, then
  the formula in the last paragraph shows that $\tilde \Delta_E =
  \Delta_{\tilde E} = 0$. So $E$ has bad reduction. Without loss of generality,
  we may assume that $\tilde \alpha = \tilde \beta$. If $\tilde \gamma = \tilde
  \alpha$, then the reduction of $E$ has the form $y^2 = (x-\tilde \alpha)^3$,
  which has a cuspidal singularity at $(\tilde \alpha, 0)$. That is, $E$ has
  additive reduction. If instead $\tilde \gamma \ne \tilde \alpha$, then $\tilde
  \gamma$ is a simple root of the reduction of the polynomial
  $(x-\alpha)(x-\beta)(x-\gamma) \in K[x]$. By Hensel's Lemma, $\gamma \in
  K$. We make the change of coordinates $x \mapsto x + \gamma$ to get an
  equation of the form
  \[
     E' \colon y^2 = x(x-\delta)(x-\varepsilon) = x^3 - (\delta +
     \varepsilon)x^2 + \delta \varepsilon x,
  \]
  where $\tilde \delta = \tilde \varepsilon$. The reduction of this equation is
  \[
     y^2 = (x - \tilde \delta)^3 + \tilde \delta (x - \tilde \delta)^2.
     \]
  Since $\tilde \delta \ne 0$, the singularity at $(\tilde \delta, 0)$ is
  nodal, and $E$ has multiplicative reduction.
\end{proof}

\begin{proposition}
  \label{prop:rational_two_torsion}
  Let $E_{/K}$ be an elliptic curve with potential multiplicative
  reduction. There is a $K$-rational $2$-torsion point $P$ such that in any
  model of $E$ over $\cO_{K_\alg}$ with multiplicative reduction, $P$ does not
  reduce to the singular point of the special fiber.
\end{proposition}

\begin{proof}
  Since $K$ does not have characteristic~2, we may begin by choosing a model
  for $E$ of the form $y^2 = f(x)$ with $f \in K[x]$ a cubic polynomial. Let
  $\alpha, \beta, \gamma \in K_\alg$ be the roots of $f$.  We claim that there
  exists a relabeling of the roots of $f$ such that
  \[
    v(\beta-\gamma) > v(\alpha - \beta) = v(\alpha - \gamma).
  \]
  That is, $\beta$ and $\gamma$ are closer to each other than to $\alpha$. To
  see it, start with any labeling of the roots. Let $u \in K_\alg$ be such that
  $v(u) = \min\Big( v(\alpha-\beta), v(\alpha - \gamma)\Big)$. Define
  \[
     g(x) = u^{-3} f(ux + \alpha) = x\left(x+\frac{\alpha-\beta}{u} \right) \left(x + \frac{\alpha - \gamma}{u}\right).
     \]
  The curve $E' \colon y^2 = g(x)$ is isomorphic to $E$ and satisfies the
  hypotheses of Lemma~\ref{lem:roots_reduction}. As $E$ has potential
  multiplicative reduction and $\tilde g$ has at least two distinct roots, it
  must have exactly two distinct roots. If $\frac{\alpha - \beta}{u}$ and
  $\frac{\alpha - \gamma}{u}$ both have valuation zero, then we already have
  the correct labeling because $v(\alpha - \beta) = v(\alpha - \gamma) = v(u)$
  and
  \[
  v(\beta - \gamma) = v\left( \frac{\alpha - \beta}{u} - \frac{\alpha-\gamma}{u}\right) + v(u) > v(u).
  \]
  Otherwise, exactly one of
  $v\Big(\frac{\alpha-\beta}{u}\Big)$ or
  $v\Big(\frac{\alpha-\gamma}{u}\Big)$ has positive valuation, say the
  former. So $v(\alpha - \gamma) = v(u)$, $v(\alpha - \beta) > v(u)$, and
  \[
  v(\beta - \gamma) = v\left( \frac{\alpha - \beta}{u} - \frac{\alpha-\gamma}{u}\right) + v(u) = v(\alpha - \gamma).
  \]
  Swapping $\alpha$ and $\gamma$ finishes the proof of the claim.

  Next we claim that $f(x)$ has a $K$-rational root. Let $L$ be a splitting
  field for $f$. Every element of $G_K$ preserves the valuations $v(\alpha -
  \beta)$, $v(\alpha - \gamma)$, and $v(\beta - \gamma)$. But the third of
  these is strictly larger than the first two, so any element of $G_K$ fixes or
  swaps $\beta$ and $\gamma$. That is, every element of $G_K$ fixes
  $\alpha$. Since $\alpha$ is a root of the separable polynomial $f$, we
  conclude $\alpha$ is $K$-rational.

  Let $P = (\alpha,0) \in E(K)$ be our distinguished 2-torsion point. On the
  model $E'$ defined earlier in the proof, the point $P$ corresponds to
  $(0,0)$. By construction, $E'$ has multiplicative reduction, and the double
  root of $\tilde g$ corresponds to the singular point of the reduced curve
  $\tilde E'$. Consequently, $(0,0)$ does not reduce to the singular
  point. This property persists under any invertible change of coordinates over
  $\cO_{K_\alg}$.
\end{proof}

\begin{proof}[Proof of Theorem~\ref{Thm:char-not-2}]
  Choose a model for $E$ of the form $y^2 = f(x)$ with $f \in K[x]$ a monic
  cubic polynomial. Write $\pi$ for a uniformizer of $K$. We consider four
  cases: the first three prove the general case of semistability, while the
  final case addresses the setting where $E$ has potential multiplicative
  reduction.

  \noindent \textbf{Case 1:} $f$ admits a $K$-rational root. After an
  appropriate translation, we obtain a new model
  \[
  E'_{/K} \colon y^2 = xg(x),
  \]
  where $g$ is monic and quadratic.  Let $m/n$ be the minimum valuation of a
  root of $g$, where $m \in \ZZ$ and $n \in \{1,2\}$. Set $L = K(\pi^{1/2n})$,
  and make the $L$-rational change of coordinates $(x,y) \mapsto (\pi^{m/n} x,
  \pi^{3m/2n}y)$ in order to obtain a third model
  \[
     E''_{/L} \colon y^2 = xh(x).
  \]
  Now $h$ is monic with integral coefficients, and at least one of the roots of
  $h$ has valuation~0. Then $E''$ has semistable reduction
  (Lem.~\ref{lem:roots_reduction}), and $L$ is totally ramified with degree
  dividing~4.

  \noindent \textbf{Case 2:} $f$ is irreducible over $K$ and the extension
  $K[x]/(f)$ is weakly totally ramified. Let $\alpha$ be a root of $f$ in this
  extension. Replacing $K$ with the cubic extension $K(\alpha)$, we find
  ourselves in Case~1. That is, $E$ attains semistable reduction after a weakly
  totally ramified extension of degree dividing~12.

  \noindent \textbf{Case 3:} $f$ is irreducible over $K$ and the extension
  $K[x]/(f)$ is unramified. Write $B$ for the minimal disk in $\CK$ containing
  the roots of $f$; it contains a $K$-rational point
  (Lem.~\ref{lem:unramified}). Without loss of generality, we may translate by
  this point in order to assume that $B$ contains~0. Irreducibility of $f$
  implies that all of its roots have the same valuation~$r$, and the unramified
  hypothesis shows that $r \in \ZZ$. Set $L = K$ if $r$ is even and $L =
  K(\pi^{1/2})$ if $r$ is odd. Rescaling by the transformation $(x,y) \mapsto
  (\pi^r x, \pi^{3r/2} y)$ gives a new model
  \[
      E'_{/L} \colon y^2 = g(x),
  \]
  where $g(x) \in \cO_L$ is monic, and all of its roots have valuation~0. The
  minimal disk about the roots of $g$ contains~0, and hence there exist roots
  $\alpha, \beta \in K_{\alg}$ such that $v(\alpha - \beta) = 0$. We conclude
  that $E'$ has semistable reduction (Lem.~\ref{lem:roots_reduction}), and 
  that $L/K$ is totally ramified of degree dividing $2$.

  \noindent \textbf{Case 4:} $E$ has potential multiplicative reduction. Let
  $P$ be a $K$-rational 2-torsion point as in
  Proposition~\ref{prop:rational_two_torsion}. Without loss of generality, we may
  translate by the $x$-coordinate of this point in order to assume that our
  model has the form
  \[
     E \colon y^2 = xg(x),
  \]
  where $g$ is a monic quadratic polynomial. Let $r$ be the minimum valuation
  of a root of $g$. We claim that $r \in \ZZ$. Assuming the claim for the
  moment, we define $L = K(\pi^{1/2})$ if $r$ is odd and $L = K$ if $r$ is
  even. Then $L/ K$ is separable and totally ramified of degree~1 or~2. Make
  the change of coordinates $(x,y) \mapsto (\pi^r x, \pi^{3r/2} y)$ to obtain
  an equation
  \[
   E'_{/L} \colon y^2 = xh(x),
  \]
  where $h$ is monic, and at least one of its roots has valuation~0. In fact,
  both roots have valuation~0, for otherwise $\tilde E'$ has a singularity at
  $(0,0)$, which contradicts our choice of $P$. We conclude that $E'$ has
  multiplicative reduction (Lem.~\ref{lem:roots_reduction}).

  Now we prove the claim that $r \in \ZZ$. Suppose not. Write $g(x) = x^2 + bx
  + c$ with $b,c \in K$. The Newton polygon for $g$ must have a single segment
  with half-integral slope. It follows that $v(c)$ is odd, and $v(b) \ge
  \frac{1}{2} v(c)$. If $\gamma$ is a root of $g$, then $v(\gamma) =
  \frac{1}{2} v(c)$, and we can rescale the original model $E$ to get a new
  model of the form
  \[
    E'' \colon y^2 = x p(x), 
  \]
  where $p(x) = \gamma^{-2} g(\gamma x)$. This model has semistable reduction
  since $xp(x)$ has at least two distinct roots over the residue field. But we
  know that $E''$ has multiplicative reduction, so we conclude that $p$ has a
  double root over the residue field. Hence, the valuation of its discriminant
  is positive:
  \[
     v\Big((b^2 - 4c) \gamma^{-2}\Big) = v(b^2 - 4c) - 2v(\gamma) = v(b^2 - 4c)
     - v(c) = v(b^2/c - 4) > 0.
     \]
  By the ultrametric inequality, we find that $v(b) = \frac{1}{2}v(c)$. But
  this is absurd since $v(c)$ is odd and $v(b) \in \ZZ$. The proof is complete.
\end{proof}


\subsection{Residue characteristic~2}
\label{sec:char2}

Throughout, assume that $K$ has residue characteristic~2. Our goal is to prove
the following result:

\begin{theorem}
\label{thm:char-2}
Let $E_{/K}$ be an elliptic curve. There is a separable weakly totally ramified
extension $L/K$ of degree dividing~24 such that $E$ has a model over $L$ with
semistable reduction. If $E$ has potential multiplicative reduction, then we
may take the degree of $L/K$ to divide~2.
\end{theorem}

An elliptic curve in this setting is given by the general Weierstrass model
\[
E_{/K} \colon y^2 + a_1 xy + a_3 y = x^3 + a_2 x^2 + a_4 x + a_6,
\]

where $a_1, a_2, a_3, a_4, a_6 \in K$. We define the standard quantities
  \begin{align*}
    b_2 &= a_1^2 + 4a_2\\
    b_4 &= 2a_4 + a_1 a_3\\
    b_6 &= a_3^2 + 4a_6\\
    b_8 &= a_1^2 a_6 + 4a_2 a_6 - a_1a_3a_4 + a_2a_3^2 - a_4^2\\
    \Delta &= -b_2^2 b_8 - 8b_4^3 - 27 b_6^2 + 9 b_2 b_4 b_6\\
    c_4 &= b_2^2 - 24b_4.
\end{align*}
Recall that if all $a_i \in \cO_K$, then $E$ has good reduction if $v(\Delta) =
0$, and $E$ has multiplicative reduction if $v(\Delta) > 0$ and $v(c_4) = 0$
\cite[Prop.~VII.5.1]{Silverman_AEC_2009}.
  
The 2-torsion of an elliptic curve over $K$ is ill-behaved under the reduction
map, so we will want to focus our attention on 3-torsion points. The 3-division
polynomial for $E$ is given by:
\begin{equation}
  \label{eq:psi3}
  \psi_3(x) = 3x^4 + b_2 x^3 + 3b_4 x^2 + 3b_6 x + b_8.
\end{equation}
The roots of $\psi_3$ are precisely the $x$-coordinates of the nontrivial
3-torsion points for $E$. We can now state an analogue of
Lemma~\ref{lem:roots_reduction} for 3-torsion.

\begin{lemma}
  \label{lem:roots_reduction2}
  Let $E_{/K}$ be an elliptic curve given by a Weierstrass model with
  coefficients in $\cO_K$. Write $\psi_3 \in \cO_K[x]$ for the 3-division
  polynomial, and write $\tilde \psi_3$ for its image in $\tilde K[x]$.
  \begin{enumerate}
  \item If $\tilde \psi_3$ has four distinct roots in an algebraic closure,
    then $E$ has good reduction.
  \item If $\tilde \psi_3$ has one simple root and one triple root in an
    algebraic closure, then $E$ has multiplicative reduction.
  \item If $\tilde \psi_3$ has a quadruple root in an algebraic closure, then
    $E$ has additive reduction.
  \end{enumerate}
  One of these three cases must occur.  In particular, this model of $E$ has
  semistable reduction if and only if the minimal disk containing its roots is $D(0,0)$.
\end{lemma}

\begin{proof}
  The statement is geometric, so we may assume the residue field $\tilde K$ is
  algebraically closed.  Let $\tilde E_{\ns}$ be the nonsingular locus of the
  reduction of $E$; it is a $\tilde K$-group scheme. Let us write $G = \tilde
  E_{\ns}[3]$ for the $3$-torsion subgroup. Observe that
  \[
  G \cong \begin{cases}
    \ZZ / 3\ZZ \times \ZZ / 3\ZZ & \text{if $E$ has good reduction} \\
    \ZZ / 3\ZZ & \text{if $E$ has multiplicative reduction} \\
    0 & \text{if $E$ has additive reduction}.
  \end{cases}
  \]
  
  If $\tilde \psi_3$ has four distinct roots, then at most one of them is the
  $x$-coordinate of a singular point. Each of the remaining roots gives rise to
  two nontrivial elements of $G$, so that $|G| \geq 6 + 1$. Hence $|G| = 9$, and $E$ has
  good reduction.

  Suppose now that $\tilde \psi_3$ has a multiple root, say $\tilde x_0$. We
  claim that $\tilde x_0$ is the $x$-coordinate of a singular point of $\tilde E$.
  One checks that $\frac{d\psi_3}{dx} = 3 \psi_2(x)$, where $\psi_2(x) = 4x^3 +
  b_2 x^2 + 2b_4 x + b_6$ is the $2$-division polynomial --- the polynomial
  whose roots are the $x$-coordinates of 2-torsion points. It follows that
  $\tilde x_0$ satisfies $\psi_2$ as well. If $P = (\tilde x_0, \tilde y_0)$
  were nonsingular on $\tilde E$, then $P$ would be both a 2-torsion point and
  a 3-torsion point, which is clearly impossible. In particular, since $\tilde
  E$ has at most one geometric singularity, we see that $\tilde x_0$ is the
  only multiple root of $\psi_3$.

  If $\tilde \psi_3$ has a simple root and a triple root, then the triple root
  corresponds to a singularity on $\tilde E$, while the simple root gives rise
  to pair of nontrivial 3-torsion points in $G$. Thus $|G| = 3$ and $E$ has
  multiplicative reduction.

  If $\tilde \psi_3$ has two simple roots and a double root, then the two simple
  roots give rise to four nontrivial 3-torsion points in $G$. This implies $|G|
  \geq 5$, so that $|G| = 9$ and $E$ has good reduction. But this is absurd
  since the double root of $\tilde \psi_3$ yields a singularity of the reduced
  curve $\tilde E$.

  Finally, suppose that $\tilde \psi_3$ has a quadruple root. Then $\tilde
  \psi_3$ has no root corresponding to a nonsingular point of $\tilde E$ and
  $|G| = 1$. Therefore, $E$ has additive reduction.
\end{proof}

Our next task is to determine how coordinate changes affect the 3-division
polynomial. The most general coordinate change on $\mathbb{A}^2$ that yields a
Weierstrass model for $E$ is
\[
(x,y) \mapsto (u^2 x + r, u^3 y + u^2 s x + t), \qquad \text{where } r,s,t \in K, u \in K^\times.
\]
Note that this transformation can be written as a composition of four basic
types of transformation:
\newpage 
\begin{itemize}
  \item \textbf{homothety}: $(x,y) \mapsto (u^2 x, u^3 y)$ for $u \in K^\times$;
  \item \textbf{$x$-translation}: $(x,y) \mapsto (x + r, y)$ for $r \in K$;
  \item \textbf{shear}: $(x,y) \mapsto (x, y + sx)$ for $s \in K$; and
  \item \textbf{$y$-translation}: $(x,y) \mapsto (x, y + t)$ for $t \in K$.
\end{itemize}

The following lemma is proved by direct computation, or by staring at the
formulas for how the $b_i$'s transform under a change of coordinates
\cite[p.45]{Silverman_AEC_2009}.

\begin{lemma}
  \label{lem:psi3_coord_change}
  Let $E_{/K}$ be an elliptic curve in Weierstrass form:
  \[
  y^2 + a_1 xy + a_3 y = x^3 + a_2x^2 + a_4x + a_6.
  \]
  The 3-division polynomial
  \[
  \psi_3(x) = 3x^4 + b_2 x^3 + 3b_4 x^2 + 3b_6 x + b_8
  \]
  is unaffected by shears and $y$-translations. Homotheties have the effect
  $\psi_3(x) \mapsto u^{-8}\psi_3(u^2 x)$, and $x$-translations have the effect
  $\psi_3(x) \mapsto \psi_3(x+r)$.
\end{lemma}

\begin{proposition}
  \label{prop:rational_three_torsion}
  Let $E_{/K}$ be an elliptic curve given by a Weierstrass equation. Suppose
  that $E$ has potential multiplicative reduction.  There is a $3$-torsion
  point $P = (\alpha,\beta) \in E(K_\alg)$ such that $\alpha \in K$, and in any
  model of $E_{K_\alg}$ with multiplicative reduction, $P$ does not reduce
  to the singular point of the special fiber.
\end{proposition}

\begin{proof}[Proof Sketch]
  Let $\psi_3$ be the $3$-division polynomial for $E$. The proof follows the
  same strategy as that of Proposition~\ref{prop:rational_two_torsion}:
  \begin{itemize}
  \item By Lemma~\ref{lem:roots_reduction2}, there is a labeling of the roots
    $\alpha, \gamma, \delta, \varepsilon$ and $c \in \QQ$ such that
    \begin{align*}
      v(\alpha - \gamma) = v(\alpha - \delta) = v(\alpha - \varepsilon) &= c \\
      v(\gamma - \delta), v(\gamma - \varepsilon), v(\delta - \varepsilon) &> c;
    \end{align*}
  \item Since $G_K$ acts by isometries, $\alpha$ must be fixed, and hence lies
    in $K$; and
  \item By the proof of Lemma~\ref{lem:roots_reduction2}, after a change of
    coordinates that makes $c = 0$, any point with $x$-coordinate $\alpha$ is
    3-torsion and reduces to a smooth point of $E$.
  \end{itemize}
  We leave the details to the reader.
\end{proof}

\begin{proof}[Proof of Theorem~\ref{thm:char-2}]
  Assume that $E$ is given by a general Weierstrass equation over $K$:
  \[
     E \colon y^2 + a_1 xy + a_3 y = x^3 + a_2 x^2 + a_4 x + a_6.
     \]
  Let $\psi_3$ be the $3$-division polynomial for $E$, and let $\pi$ be a
  uniformizer for $K$.

  \noindent \textbf{Case 1:} $E$ has a $K$-rational $3$-torsion point. We will
  show that $E$ has multiplicative reduction over $K$, or else it has a model
  with good reduction over $K(\pi^{1/3})$.

  Writing $P = (x_0, y_0)$ for the $K$-rational 3-torsion point, we make the
  change of coordinates $(x,y) \mapsto (x + x_0, y + y_0)$ in order to assume
  that $P = (0,0)$, and consequently, $a_6 = 0$. Also, $a_3 \ne 0$, for
  otherwise there would only be one $3$-torsion point with
  $x$-coordinate~$0$. We now make the change of coordinates $(x,y) \mapsto (x,
  y + a_4/a_3 x)$ in order to adjust the tangent line at $(0,0)$ to be $y = 0$.
  Since $(0,0)$ is a 3-torsion point, the line $y = 0$ is a flex. Now we have
  the equation
  \[
     E'_{/K} \colon y^2 + a xy + b y = x^3,
  \]
  with $a, b \in K$. The discriminant of the curve $E'$ is
  \[
     \Delta' = b^3(a^3 - 27 b).
     \]
     
  If $v(a^3) \le v(b)$, we apply the rescaling $(x,y) \mapsto (a^2 x, a^3y)$ to
  obtain the equation
  \[
     E''_{/K} : y^2 + xy + (b/a^3) y = x^3.
   \]
  This model has integral coefficients and $c_4$-invariant $1 - 24b / a^3$,
  which has valuation~0. That is, $E''$ has good reduction or multiplicative
  reduction (over $K$).

  If $v(a^3) > v(b)$, write $m = v(b)$ and apply the rescaling $(x,y) \mapsto
  (\pi^{2m/3} x, \pi^my)$ to obtain an integral equation with discriminant
  valuation~0 --- i.e., a model over $K(\pi^{1/3})$ with good reduction.

  \noindent \textbf{Case 2:} $E$ has potential multiplicative reduction. Let $P
  = (x_0,y_0)$ be a 3-torsion point with $x_0 \in K$ as in
  Proposition~\ref{prop:rational_three_torsion}. Replacing $K$ with $K(y_0)$,
  we are in Case~1 of the proof, and we conclude that $E$ has multiplicative
  reduction over $K(y_0)$. If $K(y_0) / K$ is weakly totally ramified, then set
  $L = K(y_0)$. Otherwise, $K(y_0)/K$ is unramified, and we set $L = K$. Since
  the reduction type does not change under unramified extension, it follows
  that $E$ has a model over $K$ with multiplicative reduction. In either case,
  the degree of $L/K$ divides~2.

  \noindent \textbf{Case 3:} $\psi_3$ has a $K$-rational root. We claim that
  $E$ attains semistable reduction over a separable weakly totally ramified
  extension of degree dividing~6. Let $P = (x_0,y_0)$ be a $3$-torsion point on
  $E$ with $x_0 \in K$. Set $K' = K(y_0)$. If $K' / K$ is weakly totally
  ramified, then we find ourselves in Case~1 with $K$ replaced by $K'$. If
  instead $K' / K$ is unramified, then the proof of Case~1 shows that $E$ has a
  model over $K'(\pi^{1/3})$ with semistable reduction. But the reduction type
  is unaffected by unramified extensions, so we also have semistable reduction
  over the field $K(\pi^{1/3})$.

  \noindent \textbf{Case 4:} $\psi_3$ has a root defined over a weakly totally
  ramified extension of $K$. Let $K' / K$ be a (separable) weakly totally
  ramified extension that contains a root of $\psi_3$. We may assume that the
  degree $[K' : K]$ divides~4. Replacing $K$ with $K'$, we find ourselves in
  Case~3, and hence $E$ attains semistable reduction over a separable weakly
  totally ramified extension of degree dividing~24.

  \noindent \textbf{Case 5:} $\psi_3$ has no root defined over a weakly totally
  ramified extension of $K$. In particular, $E$ has potential good reduction
  (Prop.~\ref{prop:rational_three_torsion}), so the roots of $\psi_3$ are
  equispaced (Lem.~\ref{lem:roots_reduction2}). Let $B$ be the minimal disk
  about the roots of $\psi_3$. There is $r \in B$ such that $K(r) / K$ is
  separable and weakly totally ramified of degree dividing~2, and such that $m
  := v(\alpha - \beta) \in v(K(r)^\times)$ for any pair of distinct roots
  $\alpha, \beta$ of $\psi_3$ (Lem.~\ref{lem:equispaced}). Then $B =
  D(\alpha,m)$. Let $u \in K_\alg$ be such that $v(u) = m/2$ and $u^2 \in
  K(r)$; if $m \in 2 v(K(r)^\times)$, we insist that $u \in K(r)$. The
  extension $K(u,r)/K$ is weakly totally ramified of degree dividing~4. Making
  the change of coordinates $(x,y) \mapsto (u^2x + r, u^3 y)$, we obtain a new
  model
  \[
    E'_{/K'} \colon y^2 + a_1' xy + a_3' y = x^3 + a_2' x^2 + a_4' x + a_6',
  \]
  where $K' = K(u,r)$. By construction, the minimal disk about the roots of the
  3-division polynomial $\psi_{3,E'}$ is $D(0,0)$. We now have two subcases;
  the strategy in each case is essentially the same, though the details are
  quite different.

  \noindent \textbf{Case 5a:} $K$ has characteristic~0. Without loss of
  generality, we may assume that $a_1' = a_3' = 0$ after making the change of
  coordinates
    $(x,y) \mapsto (x, y - a_1'/2 x - a_3' / 2)$.
  Since this is a composition of a shear and a $y$-translation, it does not
  affect the 3-division polynomial (Lem.~\ref{lem:psi3_coord_change}). We have
  \[
  \psi_{3,E'}(x) = 3x^4 + 4a_2' x^3 + 6 a_4' x^2 + 12 a_6' x + \left(4 a_2' a_6' - (a_4')^2\right).
  \]
  Define
  \[
  A = 4 a_2', \quad B = 2 a_4', \quad C = 4 a_6'.
  \]
  Then $A,B,C \in \cO_{K'}$ since we have arranged for $\psi_{3,E'}$ to have
  integral coefficients. Also, $\psi_{3,E'}$ has four distinct roots over
  $k_{\alg}$, so at least one of $A, C$ has valuation~0. Moreover, the constant
  coefficient of $\psi_{3,E'}$ must have valuation~0. For if not, $\tilde
  \psi_{3,E'}$ would have $0$ as a simple root, which would lift to a root of
  $\psi_{3,E'}$ in $K'$, contradicting the fact that $\psi_{3,E}$ does not have
  a root defined over a weakly totally ramified extension of $K$. It follows that
  \[
   v(AC - B^2) = v\left(16 a_2' a_6' - 4 (a_4')^2\right) = 2v(2). 
  \]

  Suppose first that $A$ has valuation~0. Since $v(a_2') = - 2v(2)$, the
  minimal disk about the roots of $z^2 - a_2'$ is $D := D(\sqrt{a_2'}, 0)$. We
  use Ramified Approximation to choose $s \in D$ such that $K'(s) / K'$ is
  weakly totally ramified and at worst quadratic. Next, define $t = a_4' /
  (2s)$. Note that $v(s) = -v(2)$ and $v(t) = v(a_4') \ge -v(2)$. Set $L =
  K'(s)$. Making the change of coordinates $(x,y) \mapsto (x, y+sx + t)$, we
  obtain the new model
  \[
    E''_{/L} \colon y^2 + a_1''xy + a_3''y = x^3 + a_2'' x^2 + a_6'',
  \]
  where
  \[
  a_1'' = 2s, \qquad a_2'' = a_2' - s^2, \qquad a_3'' = 2t, \qquad a_6'' = a_6' - t^2.
  \]
  Evidently, $v(a_1'') = 0$ and $v(a_3'') \ge 0$. If we write $s = \sqrt{a_2'}
  + u$ with $v(u) \ge 0$, then $s^2 = a_2' + u\sqrt{A} + u^2$. In particular,
  we find $v(a_2'') \ge 0$ and 
  \[
     v(a_6'') = v\left( \frac{AC - B^2}{16s^2} + \frac{C(u\sqrt{A} + u^2)}{4s^2}\right).
     \]
  Both displayed terms have nonnegative valuation, so $v(a_6'') \ge 0$. Thus,
  $E''$ is an integral model, and it has good reduction since we haven't
  changed the $3$-division polynomial (Lem.~\ref{lem:roots_reduction2}). The
  extension $L/K$ is weakly totally ramified of degree dividing~8.

  If instead $C$ has valuation~0, then the minimal disk about the roots of $z^2
  - a_6'$ is $D(\sqrt{a_6'}, 0)$. We choose $t$ inside this disk using the
  Ramified Approximation theorem, and we set $s = a_4' / (2t)$. After the
  change of coordinates $(x,y) \mapsto (x, y + sx + t)$, the argument is
  similar to the previous paragraph.

  \noindent \textbf{Case 5b:} $K$ has characteristic~2. The 3-division
  polynomial is given by
  \[
    \psi_{3,E'}(x) = x^4 + (a_1')^2 x^3 + a_1' a_3' x^2 + (a_3')^2 x + b_8'.
  \]
  The reduction $\tilde \psi_{3,E'}$ has four distinct roots, so at least one
  of $a_1', a_3'$ has  valuation~0. Also, $v(b_8')
  = 0$ since otherwise $\psi_{3,E}$ has a $K'$-rational root (cf. Case 5a).

  Suppose first that $v(a_1') = 0$. The roots of $z^2 + a_1' z + a_2'$ are
  $\alpha$ and $\alpha + a_1'$, so that the valuation of their difference
  is~0. That is, the minimal disk about the roots of this polynomial is $D :=
  D(\alpha, 0)$. Apply Ramified Approximation to obtain $s \in D$ with
  $K'(s) / K'$ weakly totally ramified and at worst quadratic. Set $t = (a_4' +
  s a_3') / a_1'$. Set $L = K'(s)$. Making the change of coordinates $(x,y)
  \mapsto (x, y+sx + t)$ gives the new model
  \[
     E''_{/L} \colon  y^2 + a_1'xy + a_3'y = x^3 + a_2'' x^2 + a_6'',
  \]
  where
  \[
  a_2'' = s^2 + a_1' s + a_2', \qquad a_6'' = a_6' + ta_3' + t^2.
  \]
  Writing $s = \alpha + u$ for some $u$ with $v(u) \ge 0$, we see that
  $v(a_2'') = v\left(u(u+a_1')\right)$, so that $v(a_2'') \ge 0$. This change
  of coordinates is a composition of a shear and a $y$-translation, so $b_8' =
  b_8''$ (Lem.~\ref{lem:psi3_coord_change}). Hence,
  \[
     0 = v(b_8'') = v\left( (a_1')^2 a_6'' + a_2'' (a_3')^2\right).
  \]
  Since $a_2'' (a_3')^2$ is integral and $a_1'$ has valuation~0, we conclude
  that $a_6''$ is integral.  That is, the model $E''$ is integral, and it has
  good reduction since we haven't changed the $3$-division polynomial
  (Lem.~\ref{lem:roots_reduction2}). We also see that $[L : K]$ divides~8,
  and $L/K$ is weakly totally ramified.

  If instead $v(a_3') = 0$, then the minimal disk about the roots of $z^2 +
  a_3' z + a_6'$ is $D(\beta,0)$ for $\beta$ a root. We choose $t$ inside this
  disk with Ramified Approximation, and we set $s = (a_4' + t a_1') / a_3'$.
  After the change of coordinates $(x,y) \mapsto (x, y + sx + t)$, the argument
  mimics the previous paragraph.
\end{proof}


\section{Dynamical systems on \texorpdfstring{$\PP^1$}{P1}}
\label{sec:dynamics}

Let $f \in K(z)$ be a rational function of degree $d \geq 2$. Write $f =
f_0/f_1$, where $f_0, f_1$ are coprime polynomials with integral
coefficients. After rescaling by a common element of $K^\times$, we may assume
that some coefficient of $f_0$ or $f_1$ lies in $\cO_K^\times$. Let $F_0, F_1
\in \cO_K[X,Y]$ be homogeneous polynomials of degree~$d$ such that $f_i =
F_i(z,1)$. Write $\Res(F_0,F_1)$ for the resultant of $F_0, F_1$; it is a
homogeneous polynomial of degree $2d+2$ in the coefficients of $F_0, F_1$
\cite[\S2.4]{Silverman_Dynamics_Book_2007}. Following Rumely, we write
$\ordres(f)$ for $v(\Res(F_0,F_1))$. This quantity is independent of the choice
of $F_0, F_1$, and necessarily $\ordres(f) \ge 0$. We say $f$ has
\textbf{semistable reduction} if
\[
\ordres(f)
= \min_{\sigma \in \PGL_2(\CK)} \ordres\left(\sigma^{-1} \circ f \circ \sigma\right).
\]
We say $f$ has \textbf{good reduction} if $\ordres(f) = 0$.

The \textbf{Gauss point} of $\Aff{\CK}$ is the type~II point
$\zeta_{0,0}$. Given any other type~II point $\zeta \in \Aff{\CK}$, there is an
element $\sigma \in \PGL_2(\CK)$ such that $\sigma(\zeta_{0,0}) = \zeta$, and
this $\sigma$ is unique up to left-multiplication by elements of
$\PGL_2(\cO_{\CK})$. Define $f^\sigma = \sigma^{-1} \circ f \circ \sigma$. Then
we obtain a function $\zeta \mapsto \ordres(f^\sigma)$ that does not depend
on the choice of $\sigma$, and we can extend it to a continuous piecewise
affine map $\Aff{\CK} \to \Rinf$. Rumely showed that this map is convex on
paths in $\Aff{\CK}$, and he defined the \textbf{minimal resultant locus} ---
denoted $\MRL(f)$ --- to be the subset of $\Aff{\CK}$ on which the function
$\zeta \mapsto \ordres(f^\sigma)$ attains its minimum value
\cite[Thm.~1.1]{Rumely_Min_Res_Loc}. The type~II points of the minimal
resultant locus correspond to coordinate changes $f^\sigma$ that are
semistable \cite[Thm.~C]{Rumely_new_equivariant}.

\begin{proof}[Proof of Theorem~B]
  Suppose first that $f$ has potential good reduction. Benedetto reduced this
  case to a particular refinement of Ax's Lemma in Conjecture~2 of
  \cite[\S5]{Benedetto_potential_good_bounds}. Corollary~\ref{cor:Ax_refinement}
  is that refinement.

  Now suppose that $f$ is a general rational function of degree $d \ge 2$ on
  $\PP^1_K$. We follow Rumely's argument in
  \cite[Thm.~3.2]{Rumely_Min_Res_Loc}, making the necessary
  modifications. Define $a = f(\infty)$, and let $\Fix(f)$ be the set of type~I
  fixed points of $f$ in $\bP{\CK} = \Aff{\CK} \cup \{\infty\}$. The minimal
  resultant locus is contained in the convex hull of $\Fix(f) \cup f^{-1}(a)$,
  and it is either a single type~II point or a segment with type~II endpoints
  \cite[Thm.~1.1]{Rumely_Min_Res_Loc}. Since $f$ is defined over $K$, the set
  $\MRL(f)$ is stable under the action of $G_K$, though not necessarily
  pointwise fixed.

  Let $Q \in \MRL(f)$ be a point that is fixed by $G_K$, which is guaranteed to
  exist by \cite[Prop.~3.4]{Rumely_Min_Res_Loc}. Then $Q$ corresponds to a
  $G_K$-invariant disk $D$. Since $\MRL(f)$ is contained in the convex hull of
  $\Fix(f) \cup f^{-1}(a)$, there exists $P \in \Fix(f) \cup (f^{-1}(a)
  \smallsetminus \{\infty\})$ such that $Q$ lies on the segment $[P,
    \infty]$. That is, $D$ contains $P$. Set $\delta = [K(P) : K]$. If $P \in
  \Fix(f)$, then $\delta \le d+1$, while if $P \in f^{-1}(a) \smallsetminus
  \{\infty\}$, then $\delta \le d-1$. Since $D$ is $G_K$-invariant, every
  conjugate of $P$ lies in $D$. By Corollary~\ref{cor:Ax_refinement}, there is
  an element $\alpha \in K_\alg \cap D$ such that $K(\alpha) / K$ is separable
  and weakly totally ramified, and $[K(\alpha) : K] \le q(\delta) \le
  q^*(d+1)$.

  Now $Q$ lies on the segment $[\alpha, \infty]$. Let $Q_0 \in \MRL(f)$ be the
  closest point to $\alpha$; i.e., $[\alpha,Q_0] \cap \MRL(f) = \{Q_0\}$. The
  function $\ordres(\cdot)$, when restricted to $[\alpha, \infty]$, has a break
  in linearity at $Q_0$. The discussion after (2.9) in
  \cite{Rumely_Min_Res_Loc} shows there are coprime integers $m,n$ with $1 \le
  n \le d+1$ such that $Q_0 = \zeta_{\alpha,m/n}$. (N.B. --- We are working
  with valuations, whereas Rumely works with absolute values.) Taking a
  uniformizer $\pi$ for $K(\alpha)$ and an element $c \in K$ with $v(c)$
  sufficiently large, we find that $z^n + cz + \pi^m$ is separable and has a
  root $\beta$ with valuation $m/n$. It follows that $Q_0$ is defined over
  $K(\alpha,\beta)$, which is separable and weakly totally ramified of degree
  at most $(d+1) \ q^*(d+1)$. Now conjugate $f$ by $z \mapsto \beta z + \alpha$
  in order to obtain a model with minimal resultant, and hence semistable
  reduction.
\end{proof}


\bibliographystyle{plain}
\bibliography{xander_bib}

\providecommand\biburl[1]{\texttt{#1}}
\begin{thebibliography}{10}

\bibitem{Ax_zeros_of_polynomials}
James Ax.
\newblock Zeros of polynomials over local fields --- the {G}alois action.
\newblock {\em J. Algebra}, 15:417--428, 1970.

\bibitem{Baker-Rumely_BerkBook_2010}
Matthew Baker and Robert Rumely.
\newblock {\em Potential theory and dynamics on the {B}erkovich projective
  line}, volume 159 of {\em Mathematical Surveys and Monographs}.
\newblock American Mathematical Society, Providence, RI, 2010.

\bibitem{Benedetto_potential_good_bounds}
Robert~L. Benedetto.
\newblock Attaining potentially good reduction in arithmetic dynamics.
\newblock {\em Int. Math. Res. Not.}, 22:11828--11846, 2015.

\bibitem{Berkovich_Spectral_Theory_1990}
Vladimir~G. Berkovich.
\newblock {\em Spectral theory and analytic geometry over non-{A}rchimedean
  fields}, volume~33 of {\em Mathematical Surveys and Monographs}.
\newblock American Mathematical Society, Providence, RI, 1990.

\bibitem{Bosch_et_al_Neron_Models_1990}
Siegfried Bosch, Werner L{\"u}tkebohmert, and Michel Raynaud.
\newblock {\em N\'eron models}, volume~21 of {\em Ergebnisse der Mathematik und
  ihrer Grenzgebiete (3) [Results in Mathematics and Related Areas (3)]}.
\newblock Springer-Verlag, Berlin, 1990.

\bibitem{MO:17846}
Pete~L. Clark.
\newblock Existence of maximal totally ramified extensions of an arbitrary
  cdvf.
\newblock MathOverflow.
\newblock URL:http://mathoverflow.net/q/17846 (version: 2010-03-11).

\bibitem{Faber_Berk_RamI_2013}
Xander Faber.
\newblock Topology and geometry of the {B}erkovich ramification locus for
  rational functions, {I}.
\newblock {\em Manuscripta Math.}, 142(3-4):439--474, 2013.

\bibitem{MacLane1}
Saunders MacLane.
\newblock A construction for absolute values in polynomial rings.
\newblock {\em Trans. Amer. Math. Soc.}, 40(3):363--395, 1936.

\bibitem{MacLane2}
Saunders MacLane.
\newblock A construction for prime ideals as absolute values of an algebraic
  field.
\newblock {\em Duke Math. J.}, 2(3):492--510, 1936.

\bibitem{Morton_Silverman_1994}
Patrick Morton and Joseph~H. Silverman.
\newblock Rational periodic points of rational functions.
\newblock {\em Internat. Math. Res. Notices}, (2):97--110, 1994.

\bibitem{Neukirch}
J\"{u}rgen Neukirch.
\newblock {\em Algebraic number theory}, volume 322 of {\em Grundlehren der
  mathematischen Wissenschaften [Fundamental Principles of Mathematical
  Sciences]}.
\newblock Springer-Verlag, Berlin, 1999.
\newblock Translated from the 1992 German original and with a note by Norbert
  Schappacher, With a foreword by G. Harder.

\bibitem{Obus_Srinivasan_MacLane}
Andrew Obus and Padmavathi Srinivasan.
\newblock Explicit minimal embedded resolutions of divisors on models of the
  projective line.
\newblock {\em Res. Number Theory}, 8(2):Paper No. 27, 27, 2022.

\bibitem{Obus_Wewers_MacLane}
Andrew Obus and Stefan Wewers.
\newblock Explicit resolution of weak wild quotient singularities on arithmetic
  surfaces.
\newblock {\em J. Algebraic Geom.}, 29(4):691--728, 2020.

\bibitem{Rozensztajn_2020}
Sandra Rozensztajn.
\newblock On the locus of 2-dimensional crystalline representations with a
  given reduction modulo {$p$}.
\newblock {\em Algebra Number Theory}, 14(3):643--700, 2020.

\bibitem{Rumely_Min_Res_Loc}
Robert Rumely.
\newblock The minimal resultant locus.
\newblock {\em Acta Arith.}, 169:251--290, 2015.

\bibitem{Rumely_new_equivariant}
Robert Rumely.
\newblock A new equivariant in nonarchimedean dynamics.
\newblock {\em Algebra Number Theory}, 11:841--844, 2017.

\bibitem{Ruth_Thesis_2014}
Julian~Peter R\"uth.
\newblock {\em Models of Curves and Valuations}.
\newblock PhD thesis, Universit\"at Ulm, 2014.
\newblock available at \url{https://oparu.uni-ulm.de/xmlui/123456789/3302}.

\bibitem{Serre_Corps_Locaux}
Jean-Pierre Serre.
\newblock {\em Corps locaux}.
\newblock Hermann, Paris, 1968.
\newblock Deuxi{\`e}me {\'e}dition, Publications de l'Universit{\'e} de
  Nancago, No. VIII.

\bibitem{Silverman_Dynamics_Book_2007}
Joseph~H. Silverman.
\newblock {\em The arithmetic of dynamical systems}, volume 241 of {\em
  Graduate Texts in Mathematics}.
\newblock Springer, New York, 2007.

\bibitem{Silverman_AEC_2009}
Joseph~H. Silverman.
\newblock {\em The arithmetic of elliptic curves}, volume 106 of {\em Graduate
  Texts in Mathematics}.
\newblock Springer, Dordrecht, second edition, 2009.

\bibitem{stacks-project-8-22-2026}
The {Stacks Project Authors}.
\newblock The {S}tacks {P}roject.
\newblock \url{https://stacks.math.columbia.edu}.
\newblock Accessed 22 August 2026.

\bibitem{Szydlo}
Michael Szydlo.
\newblock Elliptic fibers over non-perfect residue fields.
\newblock {\em J. Number Theory}, 104(1):75--99, 2004.

\bibitem{Tate_p_divisible_1966}
J.~T. Tate.
\newblock {$p$}-divisible groups.
\newblock In {\em Proc. {C}onf. {L}ocal {F}ields ({D}riebergen, 1966)}, pages
  158--183. Springer, Berlin-New York, 1967.

\end{thebibliography}

\end{document}